\documentclass[11pt]{article}

\usepackage{amsmath,amssymb,amscd,amsthm,makeidx,txfonts,graphicx,url,bm}
\usepackage{a4wide,pstricks,xy}

\usepackage{youngtab}

\xyoption{all}
\input{xypic}

\newtheorem{theorem}{Theorem}[section]
\newtheorem{proposition}[theorem]{Proposition}
\newtheorem{corollary}[theorem]{Corollary}
\newtheorem{conjecture}[theorem]{Conjecture}
\newtheorem{claim}[theorem]{Claim}
\newtheorem{lemma}[theorem]{Lemma}

\theoremstyle{remark}
\newtheorem{remark}[theorem]{Remark}
\newtheorem{example}[theorem]{Example}

\theoremstyle{definition}
\newtheorem{definition}[theorem]{Definition}

\def\beq{\begin{eqnarray}}
\def\eeq{\end{eqnarray}}
\def\bes{\begin{eqnarray*}}
\def\ees{\end{eqnarray*}}

\def\omhat{{\bm\omega}}
\def\tauhat{{\bm \tau}}

\def\muhat{{\bm \mu}}

\def\lambdahat{{\bm \lambda}}

\def\oP{\overline{\calP}}
\def\oT{\overline{\bT}}
\def\otT{\overline{\bT}{^o}}

\def\bT{{\mathbf{T}}}
\def\tT{{\bT}^o}

\def\bbU{\mathbb{U}}
\def\C{\mathbb{C}}
\def\M{{\mathcal{M}}}
\def\bM{\mathbb{M}}

\def\calQ{{\mathcal{Q}}}
\def\calA{{\mathcal{A}}}
\def\bfX{{\bf X}}
\def\calC{\mathcal{C}}

\def\calX{{\mathcal{X}}}

\def\calU{{\mathcal{U}}}
\def\calE{{\mathcal{E}}}

\def\calF{{\mathcal{F}}}
\def\calS{{\mathfrak{S}}}

\def\bfD{{\bf D}}
\def\oSigma{{\overline{\Sigma}}}
\def\obfSigma{{\overline{\bfSigma}}}

\def\calP{\mathcal{P}}
\def\calH{\mathcal{H}}

\def\x{\mathbf{x}}
\def\y{\mathbf{y}}

\def\v{\mathbf{v}}

\def\w{\mathbf{w}}
\def\IC{\mathcal{IC}^{\bullet}_}
\def\pIC{\underline{\mathcal{IC}}^{\bullet}_}

\def\bfL{{\bf L}}
\def\bfP{{\bf P}}
\def\bfSigma{{\bf \Sigma}}
\def\H{\mathbb{H}}

\def\N{\mathbb{Z}_{\geq 0}}

\def\F{\mathbb{F}}
\def\Q{\mathbb{Q}}
\def\bfS{{\bf S}}

\def\bC{{\bf C}}
\def\calO{{\mathcal O}}
\def\calY{{\mathcal Y}}
\def\bfO{{\bf O}}

\def\Z{\mathbb{Z}}

\def\K{\mathbb{K}}

\def\gl{{\mathfrak g\mathfrak l}}

\def\bY{{\mathbb{Y}}}

\newcommand{\nc}{\newcommand}

\nc{\op}[1]{\mathop{\mathchoice{\mbox{\rm #1}}{\mbox{\rm #1}}
{\mbox{\rm \scriptsize #1}}{\mbox{\rm \tiny #1}}}\nolimits}
\nc{\al}{\alpha}

\nc{\ep}{\varepsilon} \nc{\ga}{\gamma} \nc{\Ga}{\Gamma}
\nc{\la}{\lambda} \nc{\La}{\Lambda} \nc{\si}{\sigma}
\nc{\Sig}{{\Gamma}} \nc{\Om}{\Omega} \nc{\om}{\omega}

\nc{\SL}{{\rm SL}} \nc{\GL}{{\rm GL}} \nc{\PGL}{{\rm PGL}}
\nc{\G}{{\rm G}}

\nc{\cpt}{{\op{cpt}}} \nc{\Dol}{{\op{Dol}}} \nc{\DR}{{\op{DR}}}
\nc{\B}{{\op{B}}} \nc{\Triv}{\op{Triv}} \nc{\Hod}{{\op{Hod}}}
\nc{\Log}{{\op{Log}}} \nc{\Exp}{{\op{Exp}}} \nc{\Est}{E_{\op{st}}}
\nc{\Hst}{H_{\op{st}}} \nc{\Left}[1]{\hbox{$\left#1\vbox to
  10.5pt{}\right.\nulldelimiterspace=0pt \mathsurround=0pt$}}
\nc{\Right}[1]{\hbox{$\left.\vbox to
  10.5pt{}\right#1\nulldelimiterspace=0pt \mathsurround=0pt$}}
\nc{\LEFT}[1]{\hbox{$\left#1\vbox to
  15.5pt{}\right.\nulldelimiterspace=0pt \mathsurround=0pt$}}
\nc{\RIGHT}[1]{\hbox{$\left.\vbox to
 15.5pt{}\right#1\nulldelimiterspace=0pt \mathsurround=0pt$}}

\nc{\bee}{{\bf E}} \nc{\bphi}{{\bf \Phi}}

\begin{document}

\title{Character varieties with Zariski closures of $\GL_n$-conjugacy classes at punctures }

\author{Emmanuel Letellier \\ {\it Universit\'e de Caen} \\ {\tt letellier.emmanuel@math.unicaen.fr}}

\pagestyle{myheadings}

\maketitle

\begin{abstract}

In the paper \cite{HLV} we gave a conjectural formula for the mixed Hodge polynomials of character varieties with generic semisimple conjugacy classes at punctures and we prove a formula for the $E$-polynomial. We also proved that these character varieties are irreducible  \cite{HLV2}. In this paper we extend the results of \cite{HLV}\cite{HLV2} to character varieties with Zariski closures of arbitrary generic conjugacy classes at punctures working with intersection cohomology.  We also study  Weyl group action on the intersection cohomology of the partial resolutions of these character varieties and give a conjectural formula for the two-variables polynomials that encode the trace of the elements of the Weyl group on the subquotients of the weight filtration.  Finally, we compute the generating function of the stack count of character varieties with Zariski closure of unipotent regular conjugacy class at punctures.\end{abstract}
\vspace{.2cm} 

\noindent Key words = character varieties, Hodge polynomials, intersection cohomology, character theory of finite general linear groups.
\vspace{.2cm}


\tableofcontents

\section{Introduction}

Assume given a genus $g$ compact Riemann surface $\Sigma$, a set $S=\{a_1,\dots,a_k\}\subset \Sigma$ and a tuple $\bC=(C_1,\dots,C_k)$ of conjugacy classes of $\GL_n(\C)$, we consider the space 

$$
\calU_{\overline{\bC}}=\left\{\rho\in{\rm Hom}\,\left(\pi_1(\Sigma\backslash S),\GL_n(\C)\right)\,\left|\,\rho(\gamma_i)\in\overline{C}_i\,\,\text{ for all }\, i\right\}\right.,$$
where $\pi_1(\Sigma\backslash S)$ is the fundamental group with respect to a fixed base point, $\gamma_i$, with  $i=1,\dots,k$, is a single loop around the puncture $a_i$, and $\overline{C}_i$ denotes the Zariski closure of the conjugacy class $C_i$. In general the space of orbits $\calU_{\overline{\bC}}/\GL_n$, where $\GL_n(\C)$ acts by conjugation, is not an algebraic variety and so often one consider the quotient stack $[\calU_{\overline{\bC}}/\GL_n]$  or the GIT quotient $\M_{\overline{\bC}}:=\calU_{\overline{\bC}}//\GL_n={\rm Spec}\left(\C[\calU_{\overline{\bC}}]^{\GL_n}\right)$ that parametrizes the closed orbits. In this paper we will assume that the tuple $\bC$ is \emph{generic} (except in \S \ref{unipotent}). In this case, the quotient map $\calU_{\overline{\bC}}\rightarrow\M_{\overline{\bC}}$ is a principal $\PGL_n$-bundle in the \'etale topology. 

It is known by Saito \cite{saito} that the compactly supported intersection cohomology $IH_c^*(\M_{\overline{\bC}},\C)$ is endowed with a mixed Hodge structure from which we get the mixed Hodge Poincar\'e polynomial 

$$
IH_c(\M_{\overline{\bC}};x,y,t):=\sum_{p,q,i}{\rm dim}\,\left({\rm Gr}_{p+q}^WIH_c^i(\M_{\overline{\bC}})_\C\right)^{p,q}x^py^q t^i.
$$
Note that $IH_c(\M_{\overline{\bC}};1,1,t)$ is the compactly supported Poincar\'e polynomial $\sum_i \left({\rm dim}\,IH_c^i(\M_{\overline{\bC}},\C)\right)\, t^i$.
These moduli spaces $\M_{\overline{\bC}}$ are classical objects at least when the conjugacy classes $C_1,\dots,C_k$ are all semisimple in which case $\M_{\overline{\bC}}=\M_{\bC}$ is nonsingular and $IH_c^*(\M_\bC)$ is the usual compactly supported cohomology $H_c^*(\M_\bC)$. 

\subsection{Computing the cohomology of character varieties: State of the art}The simplest case is when $k=1$ and the conjugacy class at the puncture $a=a_1$ is $C=\zeta\cdot I$ where $\zeta=e^{\frac{2\pi i d}{n}}$ is a primitive $n$-root of unity and $I$ the identity matrix. The corresponding character variety which we denote by $\M_B^d$, is known also in the literature as the \emph{Betti moduli space}. It is known from non-Abelian Hodge theory \cite{hitchin} \cite{simpson} that $\M_B^d$ is diffeomorphic to a certain moduli space of stable Higgs bundles on $\Sigma$, known as the \emph{Dolbeault moduli space} which is often denoted by $\M_{Dol}^d$. The cohomology of the $\GL_n$ (or ${\rm SL}_n$)- character varieties $\M_B^d$ and $\M_{Dol}^d$ has been the object of a very intensive study in the literature starting from Hitchin \cite{hitchin} until recent work by Hausel-Villegas \cite{hausel-villegas}, de Cataldo-Hausel-Miggliorini \cite{CHM}, Logares-Mu\~noz-Newstead \cite{LMN}, Garcia-Prada, Heinloth, Schmitt \cite{garcia2} and Chaudouard-Laumon \cite{chaud-Laumon}. To be more precise, the Poincar\'e polynomial of $\M_{Dol}^d$ has been computed by Hitchin \cite{hitchin} when $n=2$ and by Gothen \cite{gothen} when $n=3$. Hausel and Rodriguez-Villegas \cite{hausel-villegas} gave a conjectural formula for the generating function of the mixed Hodge polynomial of the $\GL_n$-character variety  $\M_B^d$. They prove their conjecture  when $n=2$ as in this case the cohomology ring of $\M_B^d$ is known by Hausel-Thaddeus \cite{HT}.  Recently Garcia-Prada, Heinloth and Schmitt \cite{garcia2} gave a theoritical algorithm  for computing the number of points of $\M_{Dol}^d$ over finite fields  (and therefore the Poincar\'e polynomial of $\M_{Dol}^d$ as the cohomology of $\M_{Dol}^d$ is pure) from which they managed to get an explicit formula for $n=4$. Motivated by  Hausel-Villegas conjecture,  Chaudouard and Laumon \cite{chaud-Laumon} investigated the counting of stable Higgs bundles in connection with Arthur-Selberg trace formula for automorphic forms. 

The diffeomorphism bewteen $\M_B^d$ and $\M_{Dol}^d$ arising from non-Abelian theory does not preserve the mixed Hodge structure for the obvious reason that, unlike $\M_B^d$,  the mixed Hodge structure on the cohomology of $\M_{Dol}^d$ is pure. However, de Cataldo-Hausel-Migliorini \cite{CHM} conjectured that the weight filtration on $H^*(\M_B^d)$ should correspond to the perverse filtration on $H^*(\M_{Dol}^d)$ which is defined from the Hitchin map. They prove the conjecture for $n=2$. 

 In the case where the conjugacy classes $C_1,\dots,C_k$ are semisimple,  the character variety $\M_\bC$ is diffeomorphic to a certain moduli space of semistable parabolic Higgs bundles on $\Sigma$. The Poincar\'e polynomial of the later space was computed for $n=2$ by Boden and Yokogawa \cite{boden} and for $n=3$ by Garcia-Prada, Gothen and Mun\~oz \cite{garcia}. In \cite{HLV} we give a conjectural formula for the generating function of the mixed Hodge polynomial of $\M_\bC$ in terms of Macdonald polynomials. This conjecture generalizes Hausel-Villegas conjecture (the Macdonald polynomials appear trivially in Hausel-Villegas conjecture). The result of Garcia-Prada, Gothen and Mun\~oz confirms  our conjecture when $n=3$. We also prove \cite{HLV} that our conjecture is true after the specialization $t\mapsto -1$, giving thus an explicit formula for the so-called $E$-polynomials $E(\M_\bC;x,y):=H_c(\M_\bC;x,y,-1)$. We use this formula to prove that the variety $\M_\bC$ is connected (and also irreducible as it is non-singular) which is by far the most technical proof of \cite{HLV2}. 

\subsection{Results of the paper}

Unless specified the letter $\K$ denotes either $\C$ or an algebraic closure $\overline{\F}_q$ of a finite field $\F_q$ with $q$ elements.  

\subsubsection{Basic properties of character varieties} As in \S \ref{ne} we define from any tuple $\bC=(C_1,\dots,C_k)$ of conjugacy classes of $\GL_n(\K)$ a comet-shaped quiver $\Gamma_\bC$ equipped with a dimension vector $\v_\bC$. 

Denote by $\calU_\bC$ the open subset of $\calU_{\overline{\bC}}$ of representations $\rho\in{\rm Hom}\,(\pi_1(\Sigma\backslash S),\GL_n(\K))$ with $\rho(\gamma_i)\in C_i$ and by $\M_\bC$ the image of $\calU_\bC$ in $\M_{\overline{\bC}}$.

\begin{theorem}Assume that $\bC$ is generic (see Definition \ref{gendef}). The variety $\calU_{\overline{\bC}}$ is not empty if and only if $\v_\bC$ is a root of $\Gamma_\bC$. Moreover it is reduced to a single $\PGL_n$-orbit if and only if $\v_\bC$ is a real root.

Assume that  $\calU_{\overline{\bC}}\neq \emptyset$, then

\noindent (i) the quotient map $\calU_{\overline{\bC}}\rightarrow\M_{\overline{\bC}}$ is a principal $\PGL_n$-bundle in the \'etale topology. In particular the $\PGL_n$-orbits of $\calU_{\overline{\bC}}$ are all closed of same dimension ${\rm dim}\,\PGL_n$,

\noindent (ii) $\M_\bC$ is a dense nonsingular open subset of $\M_{\overline{\bC}}$ (in particular it is not empty),

\noindent (iii) the variety $\M_{\overline{\bC}}$ is irreducible of dimension

$$
d_\bC:=2gn^2-2n^2+2+\sum_{i=1}^k{\rm dim}\,C_i.
$$
\label{theo1-intro}\end{theorem}

It is not difficult to check that when $g>0$, then $\v_\bC$ is always an imaginary root (see for instance  \cite[Lemma 5.2.3]{HLV2}), and so by the above theorem the variety $\M_\bC$ is always not empty when $g>0$. 

The proof of the assertion (i) as well as of the fact that if not empty then $\M_\bC$ is a non-singular equidimensional variety of dimension $d_\bC$ goes exactly along the same lines as in the case where the conjugacy classes $C_1,\dots,C_k$ are all semisimple \cite{HLV}, see also Foth \cite[Proposition 3.4]{Foth} for the case $k=1$.

We reduce the proof of the other statements of the theorem to the semisimple case (which case is well-known) by using  the resolution $\bM_{\bf L,P,\sigma}\rightarrow\M_{\overline{\bC}}$ defined in \S \ref{resolchar} and by proving that there exists a generic tuple of semisimple conjugacy classes $\bfS$ of $\GL_n$ such that  $\#\,\bM_{\bf L,P,\sigma}(\F_q)=\# \M_\bfS(\F_q)$.

\bigskip

\begin{corollary}Assume that $\bC$ is generic and $\calU_{\overline{\bC}}$ is not empty.  The decomposition

$$
\M_{\overline{\bC}}=\coprod_{\bC'\unlhd \bC}\M_{\bC'}
$$
is a stratification, where $\bC'\unlhd\bC$ means that $C_i'\subset\overline{C}_i$ for all $i=1,\dots,k$.
\end{corollary}

For an irreducible algebraic variety $X$ over $\K$, denote by $\IC X$ the intersection cohomology complex (as defined by Goresky-MacPherson and Deligne) on $X$ such that its restriction to any nonsingular open subset is the constant sheaf concentrated in degree $0$, i.e. it is the complex with the constant sheaf $\kappa$ in degree $0$ and with $0$ in other degrees (here $\kappa=\overline{\Q}_\ell$ where $\ell$ is a fixed prime which does not divide the characteristic of $\K$ or $\kappa=\C$ if $\K=\C$). Say that $X$ is \emph{rationally smooth} if $\IC X$ is isomorphic to the constant sheaf on $X$ concentrated in degree $0$ in which case we have $H_c^*(X,\kappa)=IH_c^*(X,\kappa)$.

Denote by $i$ the inclusion of $\calU_{\overline{\bC}}$ in 

$$
\calY_{\overline{\bC}}:=(\GL_n)^{2g}\times\overline{C}_1\times\cdots\times\overline{C}_k.
$$

The following theorem will be important for the proof of the main theorem of this paper (see Theorem \ref{mainth-intro} below).

\begin{theorem}We have 

$$
i^*\left(\IC {\calY_{\overline{\bC}}}\right)\simeq\IC {\calU_{\overline{\bC}}}.
$$
\end{theorem}

The proof of the theorem uses also the resolution $\bM_{\bf L,P,\sigma}\rightarrow\M_{\overline{\bC}}$. 

Recall that Zariski closure of regular conjugacy classes are rationally smooth (this can be easily deduced from the unipotent case which is well-known \cite{BM}). Hence we have the following corollary.

\begin{corollary} For a generic tuple $\bC=(C_1,\dots,C_k)$ of conjugacy classes of $\GL_n(\K)$ with $C_1,\dots,C_k$ regular or semisimple, the character variety $\M_{\overline{\bC}}$ is rationally smooth.
\end{corollary}

\subsubsection{Mixed Hodge polynomials of character varieties}

Let us first introduce some combinatorial data. Denote by $\calP$ the set of all partitions including the empty partition $\emptyset$, and by $\bT$ the set of all maps $\omega:\Z_{>0}\times(\calP\backslash\{\emptyset\})\rightarrow\N$ with finite support. It will be also convenient for us to choose a total ordering on the set $\Z_{>0}\times(\calP\backslash\{\emptyset\})$ and write the elements of $\bT$ as non-increasing sequences $(d_1,\omega^1)(d_2,\omega^2)\cdots(d_r,\omega^r)$ (this is for instance useful  when defining the  dimension vector of some comet-shaped quiver  from multi-types as in \S \ref{ne}). Denote by $\bT_n$ the subset of types $\omega\in\bT$ of size $|\omega|:=\sum_{(d,\lambda)\in{\rm Supp}(\omega)}d\,|\lambda|=n$. The set $\bT_n$ parametrizes the types of conjugacy classes of $\GL_n(\F_q)$ (see \S \ref{conj-type}) while the subset $\tT_n$ of types $\omega\in\bT_n$ such that $(d,\lambda)\notin{\rm Supp}(\omega)$ unless $d=1$ parametrizes the types of the conjugacy classes of $\GL_n$ over an algebraically closed field (see \S \ref{ne}). Under this parametrization, types of semisimple conjugacy classes   corresponds to the elements  $\omega\in\bT$ with $\omega(d,\lambda)=0$ unless $\lambda$ is of the form $(1^m)$. We the denote by $\oT$ (resp. $\otT$) the set of elements $(\omega_1,\dots,\omega_k)$ in $\bT^k$ (resp. in $(\tT)^k$) such that $|\omega_1|=|\omega_2|=\cdots=|\omega_k|$.

In this paper (see \S \ref{def-H})  we define a  rational function $\H_\omhat(z,w)$ from Macdonald symmetric funtions for any multi-type $\omhat\in\oT$. It  is invariant under the transformations $(z,w)\mapsto (w,z)$ and $(z,w)\mapsto (-z,-w)$. For multi-types $\omhat\in\otT$ of semisimple conjugacy classes, which can be thought as multi-partitions via the map $(1,(1^{n_1}))(1,(1^{n_2}))\cdots(1,(1^{n_r}))\mapsto (n_1,\dots,n_r)$, this function was first defined in \cite{HLV}. 

Our main conjecture is the following one (see \cite{HLV} for semisimple conjugacy classes).

\begin{conjecture} Assume that $\bC$ is a generic tuple of conjugacy classes of $\GL_n(\C)$. Then

(i) the mixed Hodge polynomial $IH_c(\M_{\overline{\bC}}\,;\,x,y,t)$ depends only on $xy$ and $t$,

(ii) we have 

$$
IH_c(\M_{\overline{\bC}}\,;\, q,t):=IH_c(\M_{\overline{\bC}}\,;\,\sqrt{q},\sqrt{q},t)=(t\sqrt{q})^{d_\bC}\H_\omhat\left(-\frac{1}{\sqrt{q}},t\sqrt{q}\right),$$
where $\omhat\in\otT$ is the type of $\bC$.
\label{mainconj-intro}\end{conjecture}

Note that $IH_c(\M_{\overline{\bC}}\,;\, q,t)$ is the mixed Poincar\'e polynomial $\sum_{m,i}\left({\rm dim}\,{\rm Gr}_m IH_c^i(\M_{\overline{\bC}})\right) q^{m/2} t^i$ that encodes the weight filtration (this polynomial makes also sense in positive characteristic where the weight filtration is defined in terms of Frobenius eigenvalues).

One of the main theorem of the paper is the following one.

\begin{theorem} The above conjecture is true after the specialization $t\mapsto -1$, namely the $E^{ic}$-polynomial $E^{ic}(\M_{\overline{\bC}}\,;\, x,y):=IH_c(\M_{\overline{\bC}}\,;\, x,y,-1)$ depends only on $xy$ and 

$$
E^{ic}(\M_{\overline{\bC}}\,;\, q):=E^{ic}(\M_{\overline{\bC}}\,;\, \sqrt{q},\sqrt{q})=q^{d_\bC/2}\H_\omhat\left(\frac{1}{\sqrt{q}},\sqrt{q}\right).
$$
\label{mainth-intro}\end{theorem}

From the properties of $\H_\omhat(z,w)$ mentioned above we have the following corollary.

\begin{corollary} The polynomial $E^{ic}(\M_{\overline{\bC}}\,;\,q)$ is \emph{palindromic}, namely

$$
E^{ic}(\M_{\overline{\bC}}\,;\,q)=q^{d_\bC}E^{ic}(\M_{\overline{\bC}}\,;\, q^{-1}).
$$
\end{corollary}

\subsubsection{Partial resolutions and Weyl group actions}
 One motivation for introducing partial resolutions is to give a (conjectural) geometrical interpretation of the functions $\H_\omhat(z,w)$ for any $\omhat\in\oT$ (not necessarily in $\otT$). In \S \ref{parresol} we introduce partial resolutions $\bM_{\bf L,P,\overline{\Sigma}}\rightarrow\M_{\overline{\bC}}$ of the character variety $\M_{\overline{\bC}}$ for any generic tuple $\bC$ of conjugacy classes of $\GL_n(\K)$. There are defined from a parabolic subgroup ${\bf P}$ of $\GL_n(\K)^k$ together with a Levi factor ${\bf L}$ of ${\bf P}$ and ${\bf \Sigma}=\sigma\cdot {\bf D}$ where $\sigma$ is in the center $Z_{\bf L}$ of ${\bf L}$ and ${\bf D}$ is a unipotent conjugacy class of ${\bf L}$. We define an action of the (relative) Weyl group
 
 $$
 W({\bf L,\Sigma}):=\{g\in(\GL_n)^k\,|\, g{\bf L}g^{-1}={\bf L},\,g{\bf \Sigma}g^{-1}={\bf \Sigma})\}/{\bf L}
 $$
 on the intersection cohomology $IH_c^*(\bM_{\bf L,P,\overline{\Sigma}},\C)$ that preserves the weight filtration. 
 
 Define the $\w$-twisted mixed Poincar\'e polynomial
 
 $$
 IH_c^\w(\bM_{\bfL,\bfP,\obfSigma}\,;\,q,t):=\sum_{m,i}{\rm Tr}\left(\w,{\rm Gr}_mIH_c^i(\bM_{\bfL,\bfP,\obfSigma})\right)q^{m/2}t^i.
 $$
 
  Say that two triples $({\bf L, D,w})$ and $({\bf L', D',w'})$ with $\w\in W({\bf L,D})$ are \emph{equivalent} if there exists $g\in\GL_n(\C)^k$ such that $g({\bf L,D})g^{-1}=({\bf L',D'})$ and if  $g\w g^{-1}$ and $\w'$ are in the same $W({\bf L',D'})$-conjugacy class (here $g\w g^{-1}$ denotes the image in $W({\bf L',D'})$ of $g\dot{\w}g^{-1}$ with $\dot{\w}$ a representative of $\w$ in the normalizer $N_{(\GL_n)^k}({\bf L,D})$). The set $\oT_n$ parametrizes the equivalent classes of triples $(\bfL,\bfD,\w)$.
 
 \begin{conjecture} Assume that $\K=\C$. The mixed Hodge numbers $ih_c^{i,j;k}(\bM_{\bfL,\bfP,\obfSigma})$ are $0$ unless $i=j$ and for any $\w\in W({\bf L,\Sigma})\subset W(\bfL,\bfD)$ we have
 
  $$
 IH_c^\w(\bM_{\bfL,\bfP,\obfSigma}\,;\, q,t)=(t\sqrt{q})^{d_\bC}\H_\omhat\left(-\frac{1}{\sqrt{q}},t\sqrt{q}\right),
 $$
 where $\omhat\in\oT_n$ is the type of $(\bfL,\bfD,\w)$.
 \label{main-twisted}\end{conjecture}

We prove the following results.

\begin{proposition} Conjecture \ref{main-twisted} is equivalent to Conjecture \ref{mainconj-intro}.
\end{proposition}

\begin{theorem}Conjecture \ref{main-twisted} is true after specialization $(q,t)\mapsto (q,-1)$, namely

$$
IH_c^\w(\bM_{\bfL,\bfP,\obfSigma}\,;\,q,-1)=q^{d_\bC/2}\H_\omhat\left(\frac{1}{\sqrt{q}},\sqrt{q}\right).
$$
\label{w-twist-intro}\end{theorem}

\begin{corollary}The polynomial $E^{ic}_\w(\bM_{\bf L,P,\overline{\Sigma}}\,;\,q):=IH_c^\w(\bM_{\bfL,\bfP,\obfSigma}\,;\,q,-1)$ is palindromic.

\end{corollary}

Write $\bfL=L_1\times\cdots\times L_k$ and $\bfD=D_1\times\cdots\times D_k$ and  let $\bfS=(S_1,\dots,S_k)$ be a generic tuple of conjugacy classes of $\GL_n$ such that for each $i=1,\dots,k$, there exists an element $s_i\in S_i$ with Jordan decomposition $s_i=l_i u_i$ where $l_i$ is semisimple  such that $C_{\GL_n}(l_i)=L_i$ and $u_i\in D_i$. Then we prove the following result.

\begin{theorem}If Conjecture \ref{mainconj-intro} is true then $\M_{\overline{\bfS}}$ and $\bM_{\bfL,\bfP,\obfSigma}$ have the same mixed Hodge polynomial. \end{theorem}

Note that the above theorem implies that under Conjecture \ref{mainconj-intro} the mixed Hodge polynomial of $\bM_{\bfL,\bfP,\obfSigma}$ depends only on $(\bfL,\bfD)$ and not on the choice of $\sigma\in Z_\bfL$. In particular, taking $\sigma\in Z_{(\GL_n)^k}$ (so that $W(\bfL,\bfSigma)=W(\bfL,\bfD)$), we end up (assuming Conjecture \ref{mainconj-intro}) with an action of $W(\bfL,\bfD)$ on the intersection cohomology of $\bfS$ that preserves the weight filtration and such that for all $\w\in W(\bfL,\bfD)$

$$
\sum_{m,i}{\rm Tr}\left(\w,{\rm Gr}_mIH_c^i(\M_\bfS)\right)q^{m/2}t^i=q^{d_{\bfS}/2}\H_\omhat\left(\frac{1}{\sqrt{q}},\sqrt{q}\right),
$$
where $\omhat$ is the type of $(\bfL,\bfD,\w)$.

 Note that it is not clear how to define an action of $W(\bfL,\bfD)$ directly on the variety $\M_\bfS$. For instance in the case of semisimple conjugacy classes (i.e., $\bfD_i=\{1\}$ for all $i=1,\dots,k$) the natural action of $W(\bfL)$ on $S_1\times\cdots\times S_k$ does not preserves the relation defining character varieties.
 
 We have the following theorem.
 
 \begin{theorem} We have 
 
 $$
 E^{ic}(\bM_{\bf L,P,\overline{\Sigma}}\,;\,x)=E^{ic}(\M_{\overline{\bfS}}\,;\, x).
 $$

\label{w-twistedE-intro} \end{theorem}

\subsubsection{The regular unipotent case}

Denote by $[\M_n^{\rm uni}]$ (resp. $[\M_n^{\rm gen}]$) the quotient stack of $\calU_{\overline{\bC}}$ by $\GL_n$ where $\bC$ is the tuple of regular unipotent conjugacy classes of $\GL_n$  (resp. $\bC$ is a generic tuple of regular conjugacy classes of $\GL_n$ with only  one eigenvalue).

We prove the following relation between the two stacks counts over finite fields.

\begin{theorem}

\begin{align}
\Log\,\left(\sum_{n\geq 0}(-1)^{kn}q^{1-d_n/2}\#[\M_n^{\rm uni}](\F_q)\, T^n\right)&=\frac{1}{1-q^{-1}}\sum_{n\geq 1}(-1)^{kn}\H_{((n^1),\dots,(n^1))}\left(\sqrt{q},\frac{1}{\sqrt{q}}\right)\, T^n\label{reg}\\
&=\sum_{n\geq 1}(-1)^{kn}q^{1-d_n/2}\#[\M_n^{\rm gen}](\F_q)\, T^n
\end{align}
where $d_n=(2g+k-2)n^2-kn+2$ is the dimension of the generic character variety $\M_n^{\rm gen}$.

\end{theorem}

\subsubsection{Comments}\label{comment}

\noindent 1. Note that Theorem \ref{w-twistedE-intro} applied to $\w=1$ and $\Sigma=\{\sigma\}$ gives back (modulo Theorem \ref{Katz}) the identity 

\beq
\#\,\bM_{\bf L,P,\sigma}(\F_q)=\#\M_\bfS(\F_q)
\label{eq-intro}\eeq mentioned below Theorem \ref{theo1-intro}. In \S \ref{proofcounting}, we prove (\ref{eq-intro}) by straightforward calculation while  our proof of  Theorem \ref{w-twistedE-intro} uses intersection cohomology methods (even when $\w=1$ and $\Sigma=\{\sigma\}$). 
However, the proof of Formula (\ref{eq-intro}) we give in \S \ref{proofcounting} is more explicative and natural. 
\bigskip

\noindent 2. If $\M_{\overline{\bC}}$ is not empty, the open stratum $\M_\bC$ is always not empty by Theorem \ref{theo1-intro}(ii) but other strata $\M_{\bC'}$, with $\bC'\unlhd\bC$ maybe empty. For instance if $k=3$, $g=0$ and $\bC$ is a generic tuple of regular conjugacy classes of $\GL_2$, then the underlying graph of $\Gamma_\bC$ is the Dynkin diagram $D_4$ and $\v_\bC$ is the real root with a $2$ at the central vertex and a $1$ at the other vertices. Hence by Theorem \ref{theo1-intro} $\M_{\overline{\bC}}=\M_\bC=\{pt\}$. In particular  $\M_\bC$ is the only  stratum of $\M_{\overline{\bC}}$ which is not empty.
\bigskip

\noindent 3. There is an additive version of character varieties which we define from a \emph{generic} tuple $\bfO=(\calO_1,\dots,\calO_k)$ of adjoint orbits of $\gl_n(\C)$ as the affine GIT quotient

$$
\calQ_{\overline{\bfO}}:=\left.\left.\left\{(A_1,B_1,\dots,A_g,B_g,X_1,\dots,X_k)\in(\gl_n)^{2g}\times\prod_{i=1}^k\overline{\calO}_i\,\left|\,\sum_{j=1}^g[A_j,B_j]+\sum_{i=1}^kX_i=0\right\}\right.\right/\right/{\GL_n}.
$$

The analogue of Theorems \ref{theo1-intro}  for $\calQ_{\overline{\bfO}}$ is known  to be true. Its proof relays on a result of  Crawley-Boevey \cite{crawley-mat} \cite[Proposition 5.2.2]{letellier4} which says that $\calQ_{\overline{\bfO}}$ is a  \emph{quiver variety}. The theorem follows then from some  general results on quiver varieties due also to Crawley-Boevey (see \cite[\S 4]{letellier4} for a review) in particular irreducibility and a necessary and sufficient condition for  non-emptiness of quiver varieties. 

Our strategy to study the intersection cohomology of the character varieties $\M_{\overline{\bC}}$ follows that of $\calQ_{\overline{\bfO}}$ which is developed in \cite{letellier4}. However the intersection cohomology of $\calQ_{\overline{\bfO}}$ is simpler to study as it  always has a pure mixed Hodge structure (\cite[Theorem 7.3.2]{letellier4}) and  the mixed Hodge numbers $ih_c^{p,q;k}(\calQ_{\overline{\bfO}})$ vanish unless $p=q$. This implies that $\calQ_{\overline{\bfO}}$ has vanishing odd cohomology and that the mixed Hodge polynomial of $\calQ_{\overline{\bfO}}$ is just the Poincar\'e polynomial.  In \cite[Corollary 7.3.5]{letellier4} we prove that 

\beq
P_c(\calQ_{\overline{\bfO}};q):=\sum_i{\rm dim}\, IH_c^{2i}(\calQ_{\overline{\bC}},\C)\, q^i=q^{d_\bfO/2}\H_\omhat(0,\sqrt{q}),
\label{forJEMS}\eeq
where $\omhat\in\otT$ is the type of $\bfO$ and $\H_\omhat(z,w)$ is as above (in the semisimple case this formula was first proved in \cite{HLV}). However note that not all functions $\H_\omhat(0,\sqrt{q})$ are realized in this way as generic tuples of adjoint orbits do not exist in any type (unlike generic tuples of conjugacy classes).

If $\bC$ and $\bfO$ are respectively tuple of conjugacy classes of $\GL_n(\C)$ and tuple of adjoint orbits of $\gl_n(\C)$ of same type, then Conjecture \ref{mainconj-intro} together with Formula (\ref{forJEMS}) implies the following conjectural identity

$$
PP_c(\M_{\overline{\bC}}\,;\,q)\stackrel{?}{=}P_c(\calQ_{\overline{\bfO}}\,;\, q),
$$
where $PP_c(\M_{\overline{\bC}}\,;\,q)$ is the Poincar\'e polynomial of the pure part of the cohomology of $\M_{\overline{\bC}}$, i.e., $PP_c(\M_{\overline{\bC}}\,;\,q)=\sum_i\left({\rm dim}\,{\rm Gr}_{2i}IH_c^i(\M_{\overline{\bC}},\C)\right) \, q^i$. In the semisimple case this was first conjectured in \cite{HLV}.
\bigskip

\noindent 4. Our strategy to study cohomology of character varieties or quiver varieties (as above) go through counting points over finite fields. In the character varieties case this is related with the computation of convolution products of characteristic functions of Zariski closures of conjugacy classes while in the quiver varieties case we compute tensor products of irreducible characters of $\GL_n(\F_q)$ \cite[Theorem 1.2.2]{letellier4}. More precisely, we prove the followings. Assume that $g=0$ and fix a type $\omhat\in \oT_n$. Then for any generic tuple $(\calX_1,\dots,\calX_k)$ of irreducible characters of $\GL_n(\F_q)$ of type $\omhat$ and any generic tuple $(C_1,\dots,C_k)$ of $\F_q$-rational conjugacy classes of $\GL_n(\overline{\F}_q)$ such that $(C_1(\F_q),\dots,C_k(\F_q))$ is of type $\omhat$ we have (see Theorem \ref{convtheo} and \cite[Theorem 6.10.1]{letellier4})

\beq\xymatrix{(q-1)\left\langle {\bf X}_{\overline{C}_1}*\cdots*{\bf X}_{\overline{C}_k},1_1\right\rangle_{\GL_n(\F_q)}&\\
&\ar[ld]^{"\text{pure part}"}\ar[lu]_{t\mapsto -1}(t\sqrt{q})^{d_\bC}\H_\omhat\left(-\frac{1}{\sqrt{q}},t\sqrt{q}\right)=:{\bf H}_\omhat(q,t)\\
q^{d_\bC/2}\left\langle\calX_1\otimes\cdots\otimes\calX_k,{\rm Id}\right\rangle_{\GL_n(\F_q)}&}
\label{diag}\eeq
where roughly speaking the "pure part" is obtained by sending all terms of the form $qt^2$ to $q$ and the others to $0$ (it can be defined more rigorously in terms of specialization). Indeed the pure part of ${\bf H}_\omhat(q,t)$ is $\H_\omhat(0,\sqrt{q})$ and by \cite[Theorem 6.10.1]{letellier4}, we have $\left\langle\calX_1\otimes\cdots\otimes\calX_k,{\rm Id}\right\rangle=\H_\omhat(0,\sqrt{q})$.  Now by Conjecture \ref{main-twisted}, the function ${\bf H}_\omhat(q,t)$ is the $\w$-twisted mixed Poincar\'e polynomial of a character variety $\bM_{\bfL,\bfP,\obfSigma}$ and so the weight filtration on the intersection cohomology of character varieties  could be thought as a way to measure the default of existence of Fourier transforms on the space of class functions on $\GL_n(\F_q)$. Indeed if there were a Fourier transforms  then the two quantities  $\left\langle {\bf X}_{\overline{C}_1}*\cdots*{\bf X}_{\overline{C}_k},1_1\right\rangle_{\GL_n(\F_q)}$ and $\left\langle\calX_1\otimes\cdots\otimes\calX_k,{\rm Id}\right\rangle_{\GL_n(\F_q)}$ would be equal. In the situation of $\gl_n$ with conjugacy classes replaced by adjoint orbits and irreducible characters by characteristic functions of character sheaves, the corresponding two inner products are equal as by Lusztig \cite{LuFour} character sheaves coincide with Deligne-Fourier transforms of intersection cohomology complexes on Zariski closures of adjoint orbits. Note also that when the corresponding quiver variety $\calQ_{\overline{\bfO}}$ mentioned above  does exist, then by the results of \cite[\S 7]{letellier4} the analogue of the two above inner products in $\gl_n(\F_q)$  also equal to the Poincar\'e polynomial of $\calQ_{\overline{\bfO}}$ (which has pure mixed Hodge structure).
\bigskip

\emph{Acknowledgments}

This work is supported by the grant ANR-09-JCJC-0102-01.

\section{Preliminaries}

Unless specified, the letter $\K$ will denote either $\C$ or an algebraic closure $\overline{\F}_q$ of a finite field $\F_q$, and $\ell$ will always denote  a prime different from the characteristic of $\K$. The letter $\kappa$ will denote $\overline{\Q}_\ell$ if the characteristic of $\K$ is non-zero and $\C$ if $\K=\C$. For an algebraic variety $X/_\K$ over $\K$, we will denote by $H_c^i(X/_\K,\kappa)$ the compactly supported cohomology (this is $\ell$-adic cohomology if $\kappa=\overline{\Q}_\ell$ and the usual cohomology if $\kappa=\C$).

\subsection{Mixed Hodge polynomials and Katz theorem}

\subsubsection{Katz theorem}

Recall that  $H_c^k(X/_\C,\C)$ is endowed with a mixed Hodge structure as defined by Deligne \cite{Del1}. Namely there is a finite increasing filtration $W_\bullet^k$ on $H_c^k(X/_\C,\Q)$, called the weight filtration, such that the complexified  subquotients  $(W_r^k/W_{r-1}^k)_\C$ are endowed with a \emph{pure Hodge structure} of weight $r$. 

Denote by $\{h_c^{p,q;k}(X)\}_{p,q}$ the  mixed Hodge numbers of $H_c^k(X/_\C,\C)$ so that $\{h_c^{p,q;k}(X)\}_{p+q=r}$ are the  Hodge numbers of weight $r$, and  define the mixed Hodge polynomial  of $X/_\C$ as 

$$H_c(X/_\C\,;\,x,y,t)=\sum_{p,q,k}h_c^{p,q;k}(X)x^py^qt^k.$$

The $E$-polynomial of $X/_\C$ is defined as $$E(X/_\C\,;\,x,y):=H_c(X/_\C\,;\,x,y,-1)=\sum_{p,q}\left(\sum_k(-1)^kh_c^{p,q;k}(X)\right)x^py^q.$$

Let $R$ be a subring of $\C$ which is finitely generated as a $\Z$-algebra and let $\calX$ be a separated $R$-scheme of finite type. According to \cite[Appendix]{hausel-villegas}, we say that $\calX$ is \emph{strongly polynomial count} if there exists a polynomial $P(T)\in\C[T]$ such that for any finite field $\F_q$ and any ring homomorphism $\varphi: R\rightarrow\F_q$, the $\F_q$-scheme $\calX^{\varphi}$ obtained from $\calX$ by base change is polynomial count with counting polynomial $P$, i.e., for every finite extension $\F_{q^n}/\F_q$, we have $$\#\,\calX^{\varphi}(\F_{q^n})=P(q^n).$$

We call a separated $R$-scheme $\calX$ which gives back $X/_\C$ after extension of scalars from $R$ to $\C$ a \emph{spreading out} of $X/_\C$. 

Say that $X/_\C$ is \emph{polynomial count} if it admits a spreading out $\calX$  which is strongly polynomial count.

The following theorem is due to Katz \cite[Appendix]{hausel-villegas}.

\begin{theorem}Assume that $X/_\C$ is  polynomial count with counting polynomial $P\in\C[T]$. Then 
$$E(X/_\C\,;\,x,y)=P(xy).$$
\label{Katz}\end{theorem}

We will use the notation $E(X/_\C\,;\,x)$ for $E(X_\C\,;\,\sqrt{x},\sqrt{x})$.

\subsubsection{Intersection cohomology}
If $X$ is an equidimensional algebraic variety over $\K$, we denote by $\IC X$ the intersection cohomology complex on $X$ with coefficients in the constant sheaf $\kappa$ as defined by Goresky-MacPherson and Deligne. If $X$ is non-singular, the complex $\IC X$ is just the complex with the constant sheaf $\kappa$ in degree $0$ and $0$ in other degrees. We define the compactly supported intersection cohomology $IH_c^i(X,\kappa)$ as the compactly supported $i$-th hypercohomology group $\H_c^i(X,\IC X)$. Finally we denote by $\pIC X$  the simple  perverse sheaf on $X$ obtained by shifting by ${\rm dim}\, X$ the complex $\IC X$.
\bigskip

\noindent \textbf{Characteristic functions.} Assume that $\K$ is an algebraic closure of a finite field $\F_q$ and that $X$ is an irreducible algebraic variety defined over $\F_q$. We denote by $F:X\rightarrow X$ the corresponding Frobenius endomorphism.  Let $K$ be  in the bounded "derived" category $\mathcal{D}_c^b(X)$ of $\kappa$-(constructible) sheaves on $X$, and assume that there exists an isomorphism $\varphi: F^*(K)\simeq K$. The \emph{characteristic function} $\mathbf{X}_{K,\varphi}:X^F\rightarrow\kappa$ of $(K,\varphi)$ is defined by $$\mathbf{X}_{K,\varphi}(x)=\sum_i(-1)^i{\rm Trace}\hspace{.05cm}\big(\varphi_x^i,\mathcal{H}^i_xK\big).$$

If $Y$ is an open nonsingular $F$-stable subset of $X$, we will simply denote by ${\bf X}_{\IC X}$ the function ${\bf X}_{\IC X,\varphi}$ where $\varphi:F^*\left(\IC X\right)\rightarrow \IC X$ is the unique isomorphism which induces the identity on $\mathcal{H}^0_x\left(\IC X\right)$ for all $x\in Y^F$.
\bigskip

\noindent \textbf{Mixed Hodge polynomials.} Recall (Saito \cite{saito}, see also \cite[Chapter 14]{Peters-etal}) that  $IH_c^k(X/_\C,\C)$ is endowed with a mixed Hodge structure. Namely there is a finite increasing filtration $W_\bullet^k$ on $IH_c^k(X/_\C,\Q)$, called the weight filtration, such that the complexified  subquotients  $(W_r^k/W_{r-1}^k)_\C$ are endowed with a \emph{pure Hodge structure} of weight $r$. If $X/_\C$ is \emph{rationally smooth} (i.e. the complex $\IC X$ is isomorphic to the constant sheaf $\kappa$ concentrated in degree $0$), it coincides with Deligne's mixed Hodge structure on $H_c^k(X/_\C,\C)$ which is defined in \cite{Del1}. 

Denote by $\{ih_c^{p,q;k}(X)\}_{p,q}$ the  mixed Hodge numbers of $IH_c^k(X/_\C,\C)$ so that $\{ih_c^{p,q;k}(X)\}_{p+q=r}$ are the  Hodge numbers of weight $r$, and  define the mixed Hodge polynomial  of $X/_\C$ as 

$$IH_c(X/_\C\,;\,x,y,t)=\sum_{p,q,k}ih_c^{p,q;k}(X)x^py^qt^k.$$

The $E^{ic}$-polynomial of $X/_\C$ is defined as $$E^{ic}(X/_\C\,;\,x,y):=IH_c(X/_\C\,;\,x,y,-1)=\sum_{p,q}\left(\sum_k(-1)^kih_c^{p,q;k}(X)\right)x^py^q.$$

If $X/_\C$ is not rationally smooth, we do not get  $E^{ic}(X/_\C\,;\,q)$ by counting points over finite fields. If the $R$-scheme $\calX$ is a spreading out of $X/_\C$ and $\varphi$ is a ring homomorphism $R\rightarrow\F_q$, the best we can expect is the following formula 

\beq
E^{ic}(X/_\C\,;\,q)\stackrel{?}{=}\sum_{x\in \calX^{\varphi}(\F_q)}\bfX_{\IC {\calX^\varphi(\overline{\F}_q)}}(x).
\label{Kacinter}\eeq
Note that if $X/_\C$ is rationally smooth and polynomial count, then the right hand side of the above formula is the evaluation at $q$ of the counting polynomial of $X$. Formula (\ref{Kacinter}) is thus a  generalization of Katz formula for intersection cohomology.

Now let us \cite[Theorem 3.3.2]{letellier4} recall the conditions for Formula (\ref{Kacinter}) to hold.

\begin{definition} Let $X$ be an algebraic variety over $\K$. We say that $X=\coprod_{\alpha\in I}X_\alpha$ is a \emph{stratification} of $X$ if the set $\{\alpha\in I\,|\, X_\alpha\neq\emptyset\}$ is finite, for each $\alpha\in I$ such that $X_\alpha\neq\emptyset$, the subset $X_\alpha$ is a locally closed nonsingular equidimensional subvariety of $X$, and for each $\alpha,\beta\in I$, if  $X_\alpha\cap\overline{X}_\beta\neq\emptyset$, then $X_\alpha\subset \overline{X}_\beta$.
 \label{stratification}\end{definition}

Assume that $X=X/_\C$ is irreducible and that it has a stratification $X=\coprod_{\alpha\in I} X_\alpha$ with open stratum  $X_{\alpha_o}$. Put $\alpha\leq\beta$ if $X_\alpha\subset\overline{X}_\beta$, and  $r_\alpha:=({\rm dim}\, X_\alpha-{\rm dim}\, X)/2$. 

Say that  $X$ satisfies the property $(E)$ with respect to this stratification if there is  a spreading out $\calX$ of $X$  (say over the ring $R$), a stratification $\coprod_{\alpha\in I}\calX_\alpha$ of $\calX$ and a morphism  $\mathcal{r}:\tilde{\calX}\rightarrow\calX$ of $R$-schemes such that:

\noindent (1) $\tilde{\calX}$ and the closed strata $\calX_\alpha$ are strongly polynomial count,

\noindent (2) $\calX_\alpha$ is a spreading out of $X_\alpha$ for all $\alpha\in I$ and the morphism $r:\tilde{X}\rightarrow X$ obtained from $\mathcal{r}$ after extension of scalars from $R$ to $\C$ yields an isomorphism of mixed Hodge structures

\begin{equation}H_c^i(\tilde{X},\Q)\simeq IH_c^i(X,\Q)\oplus\left(\bigoplus_{\alpha\neq\alpha_o} W_\alpha\otimes\left(IH_c^{i+2r_\alpha}(\overline{X}_\alpha,\Q)\otimes\Q(r_\alpha)\right)\right),\label{isomhs}\end{equation}where $\Q(-d)$ is the pure mixed Hodge structure on $\Q$ of weight $2d$ and with Hodge filtration $F^d=\C$ and $F^{d+1}=0$.

\noindent (3) for any ring homomorphism $\varphi: R\rightarrow \F_q$, the morphism $\mathcal{r}^{\varphi}:\tilde{\calX}^{ \varphi}\rightarrow\calX^{\varphi}$ obtained from $\mathcal{r}$ by base change yields an isomorphism   \begin{equation}\left(\mathcal{r}^{\varphi}\right)_*(\underline{\kappa})\simeq \pIC {\calX^\varphi}\oplus\left(\bigoplus_{\alpha\neq\alpha_o} W_\alpha\otimes \pIC {\overline{\calX}_\alpha^\varphi}(r_\alpha)\right)\label{isops}\end{equation}of perverse sheaves.

Assume now that  all complex varieties $\overline{X}_\alpha$ (in particular $X$) satisfy the property $(E)$ with respect to the stratification $\overline{X}_\alpha=\coprod_{\beta\leq\alpha}X_\beta$. We have the following theorem \cite[Theorem 3.3.2]{letellier4}.

\begin{theorem} With the above assumption, there exists a polynomial $P(T)\in\Z[T]$ such that for any ring homomorphism $\varphi: R\rightarrow\F_q$, we have

\begin{equation}\sum_{x\in \calX^\varphi(\F_q)}{\bf X}_{\IC {\calX^\varphi(\overline{\F}_q)}}(x)=P(q)\label{equakatz}\end{equation}and $$E^{ic}(X\,;\,x,y)=P(xy).$$
\label{Katz2}\end{theorem}

As for the $E$-polynomial, the above theorem suggests to introduce the notation 

$$E^{ic}(X/_\C\,;\,x):=E^{ic}(X/_\C\,;\,\sqrt{x},\sqrt{x}).$$

Note that $E^{ic}(X/_\C\,;\, x)$ is then the specialization at $t=-1$ of the mixed Poincar\'e polynomial 

$$
IH_c(X/_\C\,;\,x,t):=IH_c(X/_\C\,;\,\sqrt{x},\sqrt{x},t)=\sum_{m,i}\left({\rm dim}\,{\rm Gr}_m\, IH_c^i(X/_\C)\right) x^{m/2} t^i.
$$

Note that the mixed Poincar\'e polynomial makes also sense for an algebraic variety $X$ defined over a finite field $\F_q$ where the weight filtration on $IH_c^i(X,\kappa)$ is defined in terms of Frobenius eigenvalues.

\subsection{Frobenius formula}\label{frobsec}
For technical reasons we will need to work with $\kappa=\overline{\Q}_\ell$ instead of $\C$ (although this is not required for this section where we could use $\C$). Fixing an isomorphism $\kappa\simeq\C$ gives an involution $\kappa\rightarrow\kappa,x\mapsto\overline{x}$ such that $\overline{\zeta}=\zeta^{-1}$ for any root of unity $\zeta$.

For a finite group $H$ denote by $C(H)$ the $\kappa$-space of all class functions $H\rightarrow\kappa$. The space $C(H)$ comes with two natural basis, namely the  irreducible characters of $H$ and the characteristic functions $1_C$ of conjugacy classes of $H$ that take the value $1$ on $C$ and $0$ elsewhere.  It is also endowed with two products, the pointwise multiplication (which for characters corresponds to taking tensor products)  and the convolution $*$ defined as 

$$
(f_1*f_2)(z)=\sum_{xy=z}f_1(x)f_2(y).
$$
We also define the inner product 

\beq
\langle f_1,f_2\rangle_H=\frac{1}{|H|}\sum_{h\in H}f_1(h)\overline{f_2(h)}.
\label{inner}\eeq

Given $k$-conjugacy classes $C_1,\dots,C_k$ of $H$ and an integer $g\geq 0$, define

$$
\calU_\bC=\left\{(a_1,b_1,\dots,a_g,b_g,x_1,\dots,x_k)\in H^{2g}\times \prod_i C_i\,\left|\,\prod_i(a_i,b_i)\prod_j x_j=1\right\}\right.,
$$
where $(a,b)$ is the commutator $aba^{-1}b^{-1}$.
We have the following well-known formula which goes back to Frobenius (see \cite[\S 3]{HLV})

\begin{theorem}

\beq
\frac{\#\,\calU_\bC}{|H|}=\sum_{\chi\in{\rm Irr}\, H}\left(\frac{|H|}{\chi(1)}\right)^{2g-2}\prod_{i=1}^k\frac{|C_i|\,\chi(C_i)}{\chi(1)}
\label{frobenius}\eeq

\label{frob}\end{theorem}

 It will be convenient for us to use the following straightforward formula 
\beq 
\frac{\#\calU_\bC}{|H|}=\left\langle\calE* 1_{C_1}*\cdots* 1_{C_k},1_1\right\rangle_H
\label{conv}\eeq
where 

\beq
\calE(z)=\#\left\{(a_1,b_1,\dots,a_g,b_g)\in H^{2g}\,\left|\,\prod_i(a_i,b_i)=z\right\}\right.,
\label{Xi}\eeq
and $1_1$ is the characteristic function of the trivial conjugacy class  $\{1\}$.

\section{Basic properties of character varieties}\label{sec-charvar}

\subsection{Character varieties}\label{charvarsec}

Fix a genus $g$ compact Riemann surface $\Sigma$ and a subset $S=\{a_1,\dots, a_k\}\subset\Sigma$. Consider $k$ conjugacy classes $C_1,\dots, C_k$ of $\GL_n(\K)$ and put $\bC=(C_1,\dots,C_k)$.

Define

$$
\calU_{\overline{\bC}}:=\left\{(a_1,b_1,\dots, a_g,b_g,x_1,\dots,x_k)\in (\GL_n)^{2g}\times \overline{C}_1\times\cdots\times\overline{C}_k\,\left|\, \prod_{i=1}^g(a_i,b_i)\prod_{j=1}^k x_j=I\right\}\right.
$$
where $(a,b)$ denotes the commutator $aba^{-1}b^{-1}$ and where $I$ denotes the identity matrix. Then $\calU_{\overline{\bC}}$ can be identified with the space of all representations of the fundamental group $\pi_1\left(\Sigma\backslash S\right)$ into $\GL_n(\K)$ such that the single loop around the puncture $a_i$ is sent to the Zariski closure $\overline{C}_i$ of $C_i$.

 The group $\GL_n$ acts diagonally by conjugation on $\calU_{\overline{\bC}}$ and we consider then the GIT quotient $$\M_{\overline{\bC}}:=\calU_{\overline{\bC}}//\GL_n={\rm Spec}\left(\K[\calU_{\overline{\bC}}]^{\GL_n}\right)$$ which parameterized the closed $\GL_n$-orbits of $\calU_{\overline{\bC}}$. Put $\calU_\bC:=\calU_{\overline{\bC}}\cap \left((\GL_n)^{2g}\times C_1\times\cdots\times C_k\right)$. It corresponds to  the  representations $\pi_1(\Sigma\backslash S)\rightarrow\GL_n(\K)$ mapping the single loops into $C_1,\dots,C_k$. Denote by $\M_\bC$ the image of $\calU_\bC$ in $\M_{\overline{\bC}}$. Let $(\calU_\bC)^{\rm Irr}$ be the subset of $\calU_\bC$ of irreducible representations. 
 
\begin{definition} Say that $(C_1,\dots,C_k)$ is \emph{generic} if $\prod_{i=1}^k{\rm det}\,(C_i)=1$, and if for any subspace $V\subset \K^n$ stable by some $x_i\in C_i$ (for each $i=1,\dots,k$) such that 
 $$
 \prod_{i=1}^k{\rm det}\,(x_i|_V)=1$$
 then either $V=0$ or $V=\K^n$.
 \label{gendef}\end{definition}

Our definition of generic tuple is equivalent to that given by Kostov \cite[Definition 10]{kostov}. 

\begin{example}If $\zeta\in\K$ is a primitive $n$-th root of unity, then for any unipotent conjugacy classes $C_1,\dots,C_k$, the tuple $(C_1,\dots,C_{k-1},\zeta\cdot C_k)$ is generic.\label{primitive}\end{example}

We now define a partial order on $k$-tuples of conjugacy classes of $\GL_n$ as follows. Given two $k$-tuples $\bC=(C_1,\dots,C_k)$ and $\bC'=(C_1',\dots,C'_k)$ of conjugacy classes  of $\GL_n$, say that $\bC'\unlhd \bC$ if $C_i'\subset\overline{C}{_i}$ for all $i=1,\dots,k$.

We have the following easy lemma.

\begin{lemma}Let $\bC_1$ and $\bC_2$ be two $k$-tuples of conjugacy classes of $\GL_n(\K)$ such that $\bC_1\unlhd \bC_2$. Then $\bC_1$ is generic if and only if $\bC_2$ is generic.
\end{lemma}

The lemma shows that the existence of generic tuples with prescribed Jordan form reduces to the case of semisimple conjugacy classes.

Recall the following criterion \cite[\S 2.1]{HLV}.

\begin{proposition} Assume given a $k$-tuple $\muhat=(\mu^1,\dots,\mu^k)$ of partitions of $n$ and assume that the characteristic of the field $\K$ does not divide the gcd of the parts of $\mu^1,\dots,\mu^k$. Then there always exists a generic tuple $(C_1,\dots,C_k)$ of semisimple conjugacy class of $\GL_n(\K)$ of type $\muhat$, i.e., for all $i=1,\dots,k$, the multiplicities of the eigenvalues of $C_i$ are given by the parts of the partition $\mu^i$.
\end{proposition} 

The aim of this section is to prove the following theorem.

\begin{theorem} Assume that $\bC$ is generic and that $\calU_{\overline{\bC}}\neq \emptyset$. Then 

\noindent (i) $(\calU_\bC)^{\rm Irr}=\calU_\bC$ and the quotient map $\calU_{\overline{\bC}}\rightarrow\M_{\overline{\bC}}$ is a principal $\PGL_n$-bundle in the \'etale topology. In particular the $\PGL_n$-orbits of $\calU_{\overline{\bC}}$ are all closed of same dimension ${\rm dim}\,\PGL_n$. 

\noindent (ii) $\M_\bC$ is a dense nonsingular open subset of $\M_{\overline{\bC}}$ (in particular it is not empty).

\noindent (iii) The variety $\M_{\overline{\bC}}$ is irreducible of dimension

$$
d_\bC:=2gn^2-2n^2+2+\sum_{i=1}^k{\rm dim}\,C_i.
$$

\label{geoth}\end{theorem}

When the conjugacy classes $C_1,\dots,C_k$ are semisimple, this theorem is proved in \cite[\S 2.1]{HLV} except the irreducibility which is proved in \cite{HLV2} (the proof of the irreducibility is by far the most difficult part). The assertion (i) is an easy generalisation of the semisimple case as well as the fact that, if not empty, then $\M_\bC$ is nonsingular with connected components all of same dimension $d_\bC$. Note that since $\calU_{\overline{\bC}}$ is affine, the fact that $\calU_{\overline{\bC}}\rightarrow\M_{\overline{\bC}}$ is a principal $\PGL_n$-bundle is equivalent (see Bardsley-Richardson \cite[Proposition 8.2]{BR}) to the fact that  $\PGL_n$ acts set-theoritically freely on $\calU_{\overline{\bC}}$ and that the $\PGL_n$-orbits are all separable. Only the irreducibility as well as the equivalence between the non-emptiness of $\M_\bC$ and that of $\M_{\overline{\bC}}$ require new arguments. 

The following next sections are devoted to the proof of the irreducibility of $\M_{\overline{\bC}}$ and the non-emptiness of $\M_\bC$ (assuming the non-emptiness of $\M_{\overline{\bC}}$). 

Let us state a corollary.

\begin{corollary}Assume that $\bC$ is generic and $\calU_{\overline{\bC}}$ is not empty.  The decomposition

$$
\M_{\overline{\bC}}=\coprod_{\bC'\unlhd \bC}\M_{\bC'}
$$
is a stratification.
\label{charstrat}\end{corollary}

\begin{remark} If $\M_{\overline{\bC}}$ is not empty and $\bC$ generic, then the open stratum $\M_\bC$ is also not empty  by Theorem \ref{geoth} but other strata may be empty (it is even possible to have $\M_{\overline{\bC}}=\M_\bC$ with $\bC\neq\overline{\bC}$, see \S \ref{comment} Comment 2).
\end{remark}

\subsection{Non-emptiness of $\calU_\bC$ and roots of quivers}\label{ne}

When $g=0$, the problem of describing the $k$-tuples $\bC$ for which $(\calU_\bC)^{\rm Irr}$ is not empty is stated and studied by Kostov (see \cite{kostov} for a survey) who calls it the \emph{Deligne-Simpson problem}. In \cite[Conjecture 1.4]{crawleyind}, Crawley-Boevey gives a conjectural answer of the Deligne-Simpson problem in terms of root systems of certain  quiver (which we describe below). He proves his conjecture assuming that $\bC$ is generic (in which  case we have  $(\calU_\bC)^{\rm Irr}=\calU_\bC$ by Theorem \ref{geoth}). In order to state his result (which will be needed later) and also to extend it for $g>0$ we need to explain how to construct this quiver from a $k$-tuple of conjugacy classes.

 We need to recall the notion of types for conjugacy classes \cite[\S 4.3.1]{letellier4}.

Denote by $\calP$ the set of all partitions including the unique partition $\emptyset$ of size $0$ and by $\calP_n$ the set of partitions of size $n$. A partition $\lambda\neq \emptyset$ will be written either in the form $(n_1,\dots,n_r)$ with $n_1\geq n_2\geq\cdots\geq n_r$ or in the form $(1^{m_1},2^{m_2},\dots)$ where $m_i$ denotes the multiplicity with which $i$ appears in $\lambda$. Choose a total ordering $\geq$ on $\calP$ (e.g. the lexicographic ordering) and denote by $\tT$ the set of non-increasing sequences $\omega^1\geq\omega^2\geq\cdots\geq\omega^r$ of non-zero partitions which we will denote simply by $\omega^1\omega^2\cdots\omega^r$. We denote by $\bT{^o_n}$ the subset of elements $\omhat=\omega^1\omega^2\cdots\omega^r$ of size $|\omhat|:=\sum_i|\omega^i|=n$. We define the \emph{type} of a conjugacy class $C$ of $\GL_n(\K)$ to be the sequence $\omega^1\omega^2\cdots\omega^s\in\bT{^o_n}$ where  the partitions $\omega^1,\dots,\omega^s$ are given by the size of the Jordan blocks corresponding to the distinct eigenvalues $\alpha_1,\dots,\alpha_s$ of $C$. In particular, sequences of length $1$ are the types of the conjugacy classes with only one eigenvalue, and sequences of the form $(1^{n_1})(1^{n_2})\cdots(1^{n_r})$ correspond to semisimple conjugacy classes with $r$ distinct eigenvalues of multiplicities $n_1,\dots,n_r$.

Given a type $\omega=\omega^1\omega^2\cdots\omega^r\in\tT$, we draw the Young diagrams of $\omega^1,\dots,\omega^r$ respectively from the left to the right and we also label the columns from the left to the right (with the convention that partitions are represented by the rows of the Young diagrams).

\begin{example} If our type is $(2^2)(1^2)$, the corresponding Young diagrams are 

\begin{equation}\overset{1\hspace{.3cm}2}{\yng(2,2)} \quad \overset{3}{\yng(1,1)}\label{ex}\end{equation}
\end{example}

Now let $d$ be the total number of columns (in the example above, we have $d=3$) and let $n_i$ be the length of the $i$-th column with respect to our labeling. This defines a type $A$ quiver $\Gamma$

$$\xymatrix{\bullet^1&\bullet^2\ar[l]&\cdots\ar[l]&\bullet^{d-1}\ar[l]}$$
equiped with a dimension vector $\v:=(v_1,\dots,v_{d-1})$ with $v_1=|\omega|-n_1$ and $v_i=v_{i-1}-n_i$ for $i>1$.

Hence from the type of any conjugacy class of $\GL_n$ we can define a type $A$ quiver $\Gamma_C$ equiped with a dimension vector $\v_C$. 

\begin{example}It is important to notice that the data $(\Gamma,\v)$ does not characterize the type from which it is defined. For instance any regular conjugacy class of $\GL_n$ (i.e., conjugacy classes of maximal  dimension $n^2-n$) will give the same quiver with same dimension vector, namely the type $A$ quiver with $n-1$ vertices and the dimension vector $(n-1,n-2,\dots,1)$.
\end{example}

Given a $k$-tuple $\bC=(C_1,\dots,C_k)$ of conjugacy classes of $\GL_n$, we define a multi-type in $(\bT{^o_n})^k$ as above and therefore we obtain $k$ type $A$ quivers $\Gamma_{C_1},\dots,\Gamma_{C_k}$ with  dimension vectors $\v_{C_1},\dots,\v_{C_k}$. Adding an extra node $0$ which we connect to these $k$ graphs and adding $g$ loops on this node we get the following comet-shaped quiver $\Gamma_\bC$

\vspace{10pt}

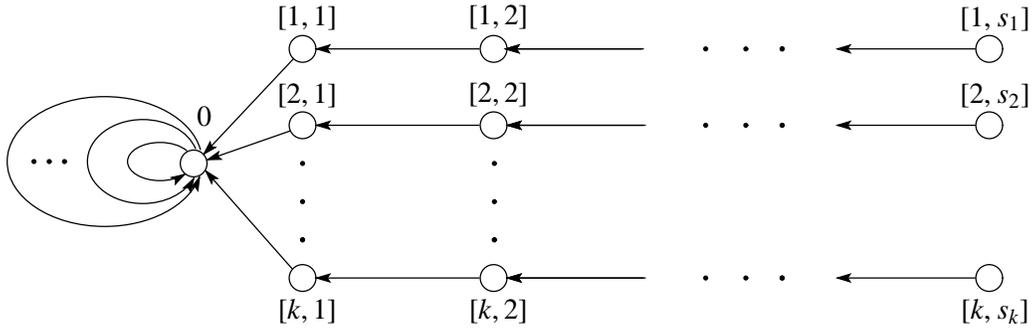
\begin{center}
\unitlength 0.1in
\begin{picture}( 52.1000, 15.4500)(  4.0000,-17.0000)
%
\special{pn 8}%
\special{ar 1376 1010 70 70  0.0000000 6.2831853}%
%
\special{pn 8}%
\special{ar 1946 410 70 70  0.0000000 6.2831853}%
%
\special{pn 8}%
\special{ar 2946 410 70 70  0.0000000 6.2831853}%
%
\special{pn 8}%
\special{ar 5540 410 70 70  0.0000000 6.2831853}%
%
\special{pn 8}%
\special{ar 1946 810 70 70  0.0000000 6.2831853}%
%
\special{pn 8}%
\special{ar 2946 810 70 70  0.0000000 6.2831853}%
%
\special{pn 8}%
\special{ar 5540 810 70 70  0.0000000 6.2831853}%
%
\special{pn 8}%
\special{ar 1946 1610 70 70  0.0000000 6.2831853}%
%
\special{pn 8}%
\special{ar 2946 1610 70 70  0.0000000 6.2831853}%
%
\special{pn 8}%
\special{ar 5540 1610 70 70  0.0000000 6.2831853}%
%
\special{pn 8}%
\special{pa 1890 1560}%
\special{pa 1440 1050}%
\special{fp}%
\special{sh 1}%
\special{pa 1440 1050}%
\special{pa 1470 1114}%
\special{pa 1476 1090}%
\special{pa 1500 1088}%
\special{pa 1440 1050}%
\special{fp}%
%
\special{pn 8}%
\special{pa 2870 410}%
\special{pa 2020 410}%
\special{fp}%
\special{sh 1}%
\special{pa 2020 410}%
\special{pa 2088 430}%
\special{pa 2074 410}%
\special{pa 2088 390}%
\special{pa 2020 410}%
\special{fp}%
%
\special{pn 8}%
\special{pa 3720 410}%
\special{pa 3010 410}%
\special{fp}%
\special{sh 1}%
\special{pa 3010 410}%
\special{pa 3078 430}%
\special{pa 3064 410}%
\special{pa 3078 390}%
\special{pa 3010 410}%
\special{fp}%
\special{pa 3730 410}%
\special{pa 3010 410}%
\special{fp}%
\special{sh 1}%
\special{pa 3010 410}%
\special{pa 3078 430}%
\special{pa 3064 410}%
\special{pa 3078 390}%
\special{pa 3010 410}%
\special{fp}%
%
\special{pn 8}%
\special{pa 2870 810}%
\special{pa 2020 810}%
\special{fp}%
\special{sh 1}%
\special{pa 2020 810}%
\special{pa 2088 830}%
\special{pa 2074 810}%
\special{pa 2088 790}%
\special{pa 2020 810}%
\special{fp}%
%
\special{pn 8}%
\special{pa 2870 1610}%
\special{pa 2020 1610}%
\special{fp}%
\special{sh 1}%
\special{pa 2020 1610}%
\special{pa 2088 1630}%
\special{pa 2074 1610}%
\special{pa 2088 1590}%
\special{pa 2020 1610}%
\special{fp}%
%
\special{pn 8}%
\special{pa 3730 810}%
\special{pa 3020 810}%
\special{fp}%
\special{sh 1}%
\special{pa 3020 810}%
\special{pa 3088 830}%
\special{pa 3074 810}%
\special{pa 3088 790}%
\special{pa 3020 810}%
\special{fp}%
\special{pa 3740 810}%
\special{pa 3020 810}%
\special{fp}%
\special{sh 1}%
\special{pa 3020 810}%
\special{pa 3088 830}%
\special{pa 3074 810}%
\special{pa 3088 790}%
\special{pa 3020 810}%
\special{fp}%
%
\special{pn 8}%
\special{pa 3730 1610}%
\special{pa 3020 1610}%
\special{fp}%
\special{sh 1}%
\special{pa 3020 1610}%
\special{pa 3088 1630}%
\special{pa 3074 1610}%
\special{pa 3088 1590}%
\special{pa 3020 1610}%
\special{fp}%
\special{pa 3740 1610}%
\special{pa 3020 1610}%
\special{fp}%
\special{sh 1}%
\special{pa 3020 1610}%
\special{pa 3088 1630}%
\special{pa 3074 1610}%
\special{pa 3088 1590}%
\special{pa 3020 1610}%
\special{fp}%
%
\special{pn 8}%
\special{pa 5466 410}%
\special{pa 4746 410}%
\special{fp}%
\special{sh 1}%
\special{pa 4746 410}%
\special{pa 4812 430}%
\special{pa 4798 410}%
\special{pa 4812 390}%
\special{pa 4746 410}%
\special{fp}%
%
\special{pn 8}%
\special{pa 5466 810}%
\special{pa 4746 810}%
\special{fp}%
\special{sh 1}%
\special{pa 4746 810}%
\special{pa 4812 830}%
\special{pa 4798 810}%
\special{pa 4812 790}%
\special{pa 4746 810}%
\special{fp}%
%
\special{pn 8}%
\special{pa 5466 1610}%
\special{pa 4746 1610}%
\special{fp}%
\special{sh 1}%
\special{pa 4746 1610}%
\special{pa 4812 1630}%
\special{pa 4798 1610}%
\special{pa 4812 1590}%
\special{pa 4746 1610}%
\special{fp}%
%
\special{pn 8}%
\special{pa 1880 840}%
\special{pa 1450 990}%
\special{fp}%
\special{sh 1}%
\special{pa 1450 990}%
\special{pa 1520 988}%
\special{pa 1500 972}%
\special{pa 1506 950}%
\special{pa 1450 990}%
\special{fp}%
%
\special{pn 8}%
\special{pa 1900 460}%
\special{pa 1430 960}%
\special{fp}%
\special{sh 1}%
\special{pa 1430 960}%
\special{pa 1490 926}%
\special{pa 1468 922}%
\special{pa 1462 898}%
\special{pa 1430 960}%
\special{fp}%
%
\special{pn 8}%
\special{sh 1}%
\special{ar 1946 1010 10 10 0  6.28318530717959E+0000}%
\special{sh 1}%
\special{ar 1946 1210 10 10 0  6.28318530717959E+0000}%
\special{sh 1}%
\special{ar 1946 1410 10 10 0  6.28318530717959E+0000}%
\special{sh 1}%
\special{ar 1946 1410 10 10 0  6.28318530717959E+0000}%
%
\special{pn 8}%
\special{sh 1}%
\special{ar 4056 410 10 10 0  6.28318530717959E+0000}%
\special{sh 1}%
\special{ar 4266 410 10 10 0  6.28318530717959E+0000}%
\special{sh 1}%
\special{ar 4456 410 10 10 0  6.28318530717959E+0000}%
\special{sh 1}%
\special{ar 4456 410 10 10 0  6.28318530717959E+0000}%
%
\special{pn 8}%
\special{sh 1}%
\special{ar 4056 810 10 10 0  6.28318530717959E+0000}%
\special{sh 1}%
\special{ar 4266 810 10 10 0  6.28318530717959E+0000}%
\special{sh 1}%
\special{ar 4456 810 10 10 0  6.28318530717959E+0000}%
\special{sh 1}%
\special{ar 4456 810 10 10 0  6.28318530717959E+0000}%
%
\special{pn 8}%
\special{sh 1}%
\special{ar 4056 1610 10 10 0  6.28318530717959E+0000}%
\special{sh 1}%
\special{ar 4266 1610 10 10 0  6.28318530717959E+0000}%
\special{sh 1}%
\special{ar 4456 1610 10 10 0  6.28318530717959E+0000}%
\special{sh 1}%
\special{ar 4456 1610 10 10 0  6.28318530717959E+0000}%
\put(19.7000,-2.4500){\makebox(0,0){$[1,1]$}}%
\put(29.7000,-2.4000){\makebox(0,0){$[1,2]$}}%
\put(55.7000,-2.5000){\makebox(0,0){$[1,s_1]$}}%
\put(19.7000,-6.5500){\makebox(0,0){$[2,1]$}}%
\put(29.7000,-6.4500){\makebox(0,0){$[2,2]$}}%
\put(55.7000,-6.5500){\makebox(0,0){$[2,s_2]$}}%
\put(19.7000,-17.8500){\makebox(0,0){$[k,1]$}}%
\put(29.7000,-17.8500){\makebox(0,0){$[k,2]$}}%
\put(55.7000,-17.8500){\makebox(0,0){$[k,s_k]$}}%
\put(14.3000,-7.6000){\makebox(0,0){$0$}}%
\special{pn 8}%
\special{sh 1}%
\special{ar 2950 1010 10 10 0  6.28318530717959E+0000}%
\special{sh 1}%
\special{ar 2950 1210 10 10 0  6.28318530717959E+0000}%
\special{sh 1}%
\special{ar 2950 1410 10 10 0  6.28318530717959E+0000}%
\special{sh 1}%
\special{ar 2950 1410 10 10 0  6.28318530717959E+0000}%
\special{pn 8}%
\special{ar 1110 1000 290 220  0.4187469 5.9693013}%
\special{pn 8}%
\special{pa 1368 1102}%
\special{pa 1376 1090}%
\special{fp}%
\special{sh 1}%
\special{pa 1376 1090}%
\special{pa 1324 1138}%
\special{pa 1348 1136}%
\special{pa 1360 1158}%
\special{pa 1376 1090}%
\special{fp}%
\special{pn 8}%
\special{ar 910 1000 510 340  0.2464396 6.0978374}%
\special{pn 8}%
\special{pa 1400 1096}%
\special{pa 1406 1084}%
\special{fp}%
\special{sh 1}%
\special{pa 1406 1084}%
\special{pa 1362 1138}%
\special{pa 1384 1132}%
\special{pa 1398 1152}%
\special{pa 1406 1084}%
\special{fp}%
\special{pn 8}%
\special{sh 1}%
\special{ar 540 1000 10 10 0  6.28318530717959E+0000}%
\special{sh 1}%
\special{ar 620 1000 10 10 0  6.28318530717959E+0000}%
\special{sh 1}%
\special{ar 700 1000 10 10 0  6.28318530717959E+0000}%
\special{pn 8}%
\special{ar 1200 1000 170 100  0.7298997 5.6860086}%
\special{pn 8}%
\special{pa 1314 1076}%
\special{pa 1328 1068}%
\special{fp}%
\special{sh 1}%
\special{pa 1328 1068}%
\special{pa 1260 1084}%
\special{pa 1282 1094}%
\special{pa 1280 1118}%
\special{pa 1328 1068}%
\special{fp}%
\end{picture}%
\end{center}

\vspace{10pt}

\noindent together with the dimenion vector $\v_\bC$ with coordinate $n$ at the vertex $0$ and with coordinates given by $\v_{C_i}$ on the legs.

To such a quiver we can associate as in \cite{kac} a root system $\Phi(\Gamma_\bC)$.

We then have the following theorem \cite[Theorem 8.3]{crawleyind}.

\begin{theorem}Assume that $\bC$ is generic and that $g=0$. Then the following are equivalent :

\noindent (i) $\calU_\bC$ is not empty.

\noindent (ii)  $\v_\bC\in\Phi(\Gamma_\bC)$.

Moroever $\calU_\bC$ is reduced to a single $\GL_n$-orbit (i.e., $\M_\bC$ is a point) if and only if $\v_\bC$ is a real root.

\label{nemp}\end{theorem}

 Note that if $g>0$, then $\v_\bC$ is always an imaginary root (actually in the fundamental domain) as explained in \cite[\S 5.2]{HLV2}. We will see (see Corollary \ref{maincoro}) that $\calU_\bC$ is always non-empty when $g>0$ and so in the above theorem the condition $g=0$ is obsolete.

\subsection{Resolutions of character varieties}\label{resolchar}

Let $P$ be a parabolic subgroup of $\GL_n(\K)$ and let $L$ be a Levi factor of $P$. Let $\sigma$ live in the center $Z_L$ of $L$ and denote by $U_P$ the unipotent radical of $P$. Put 

$$\bY_{L,P,\sigma}:=\left\{(x,gP)\in\GL_n\times (\GL_n/P)\,\left|\, g^{-1}x g\in\sigma\cdot U_P\right\}\right..$$It is well-known that the image of $\bY_{L,P,\sigma}\rightarrow\GL_n$, $(x,gP)\rightarrow x$ is the Zariski closure $\overline{C}$ of some conjugacy class $C$ of $\GL_n$ and that the map $p:\bY_{L,P,\sigma}\rightarrow\overline{C}$ is a  resolution of $\overline{C}$ (see for instance \cite{Bao} and the references therein in the case of nilpotent orbits).  The Zariski closure of any conjugacy class of $\GL_n$ has a resolution of this form.  

If $\sigma$ has a unique eigenvalue $a$ and the size of the "blocks" of $L$ gives the partitions $\mu$, then $C$ is the conjugacy class  with unique  eigenvalue $a$ and with Jordan blocks of size given by the dual partition $\mu'$ of $\mu$.

Now assume given a generic tuple $\bC=(C_1,\dots,C_k)$ of conjugacy classes of $\GL_n(\K)$ with resolutions $\bY_{L_i,P_i,\sigma_i}\rightarrow\overline{C}_i$ as above. Denote by $\bbU_{\bf L,P,\sigma}$ the subspace of elements $((a_i,b_i)_i,(x_j,g_jP_j)_j)\in(\GL_n)^{2g}\times\left(\prod_{j=1}^k\bY_{L_i,P_i,\sigma_i}\right)$ which verify the equation 

$$
(a_1,b_1)\cdots(a_g,b_g)\, x_1\cdots x_k=1.
$$

We have the following theorem.

\begin{theorem}The geometric quotient $\bbU_{\bf L,P,\sigma}\rightarrow \bM_{\bf L,P,\sigma}$ exists, is a principal $\PGL_n$-bundle in the \'etale topology and makes the following diagram Cartesian

$$
\xymatrix{\bbU_{\bf L,P,\sigma}\ar[rr]^p\ar[d]&&\calU_{\overline{\bC}}\ar[d]\\
\bM_{\bf L,P,\sigma}\ar[rr]^{p/_{\PGL_n}}&&\M_{\overline{\bC}}}.
$$
Moreover $\bM_{\bf L,P,\sigma}$ is nonsingular with connected components all of same dimension.
\label{nonsing}\end{theorem}

\begin{proof} It follows from a general theorem of Mumford \cite[Proposition 7.1]{mumford} as the morphism $p$ is projective and  the quotient $\calU_{\overline{\bC}}\rightarrow\M_{\overline{\bC}}$ is a principal $\PGL_n$-bundle for the \'etale topology. The proof of the last statement goes exactly along the same lines as the proof of the "additive version"  \cite[Proof of Theorem 5.3.7]{letellier4}.
\end{proof}

Now choose a generic tuple $\bfS=(S_1,\dots,S_k)$ of semisimple conjugacy classes of $\GL_n$ such that for each $i=1,\dots,k$, there exists a representative of $S_i$ whose stabilizer in $\GL_n$ is $L_i$.

Here is the main theorem of this section.

\begin{theorem}Let $\K=\overline{\F}_q$ and consider the  $\F_q$-structure on $\GL_n$ corresponding to the Frobenius $F:\GL_n\rightarrow\GL_n$, $(a_{ij})_{i,j}\mapsto (a_{ij}^q)_{i,j}$. For each $i=1,\dots,k$, we assume that $P_i,L_i,\sigma_i$ as well as the eigenvalues of $S_i$  are defined over $\F_q$ (i.e. fixed by the Frobenius $F$). Then $\bbU_{\bf L,P,\sigma}$ and $\calU_\bfS$ are also defined over $\F_q$ and 

$$
\# \bM_{\bf L,P,\sigma}(\F_{q^r})=\# \M_\bfS(\F_{q^r})
$$
for all integer $r>0$.
\label{counting}\end{theorem}

\begin{remark}The quiver varieties analogue of this theorem is well-known by a result of Nakajima (see \cite[Proof of Theorem 5.3.7]{letellier4} and references therein). 
\end{remark}

It is possible as explained in the appendix of \cite{HLV} to define $R$-schemes $\calX_{\bf L,P,\sigma}$ and $\calX_\bfS$ (for some subring $R$ of $\C$ which is finitely generated as a $\Z$-algebra) which are spreadings out of $\bM_{\bf L,P,\sigma}/_\C$ and $\M_\bfS/_\C$. We already know \cite[Theorem 1.2.3]{HLV} that $\M_\bfS/_\C$ is polynomial count. Hence by the above theorem $\bM_{\bf L,P,\sigma}/_\C$ is also polynomial count and by Theorem \ref{Katz} we deduce the following result.

\begin{corollary} Assume that $\K=\C$, then 

$$
E(\bM_{\bf L,P,\sigma};x)=E(\M_\bfS;x).
$$
\label{E-poly-coro}\end{corollary}

We actually conjecture (see the more general conjecture \ref{main3}) that the mixed Hodge polynomials of $\bM_{\bf L,P,\sigma}$ and $\M_\bfS$ agree.
\bigskip

\begin{corollary}If not empty, the varieties $\bM_{\bf L,P,\sigma}$ and $\M_{\overline{\bC}}$ are both irreducible. They are  non-empty if and only if the vector dimension $\v_\bC$ is a root of $\Gamma_\bC$. In particular, if $\M_{\overline{\bC}}$ is not empty, then so is $\M_{\bC}$.
\label{maincoro}\end{corollary}

\begin{proof} Let us first prove the irreducibility.  Assume first that $\K=\overline{\F}_q$. For an algebraic variety $X$ of dimension $d$  defined over $\F_q$, it is now well-known by deep results of Deligne that  the coefficient of $q^{dr}$ in $\#\, X(\F_{q^r})$ equals the number of irreducible components of $X$ of dimension $d$ and that all complex conjugates of the eigenvalues of Frobenius on $H_c^i(X,\kappa)$ with $i<2d$ have absolute value less or equal to $q^{d-\frac{1}{2}}$. Since $\M_\bfS$ is irreducible (this is the main result of \cite{HLV2}) and that the irreducible components of $\bM_{\bf L,P,\sigma}$ are all of same dimension (by Theorem \ref{nonsing}), we deduce from Theorem \ref{counting} that $\bM_{\bf L,P,\sigma}$ is also irreducible. Therefore $\M_{\overline{\bC}}$ is irreducible. If $\K=\C$ we deduce the corollary from Corollary \ref{E-poly-coro}. Indeed, as $\M_\bfS$ and $\bM_{\bf L,P,\sigma}$ are non-singular, the coefficient of the term of highest degree in $E(\M_\bfS\,;\,x)$ and $E(\bM_{\bf L,P,\sigma}\,;\,x)$ equals the dimension of the top compactly supported cohomology which is also equal to the number of irreducible components of maximal dimension.

Suppose that $\M_{\overline{\bC}}$ is not empty. Then $\bM_{\bf L,P,\sigma}$ is also not empty and so by Theorem \ref{counting}, the variety $\M_\bfS$ is not empty.  Applying again Theorem \ref{nemp} to $\bfS$, we get that $\v_\bfS$ is a root of $\Gamma_\bfS$ (recall that it is always a root if $g>0$, see below Theorem \ref{nemp}). But it is not difficult to see that $\Gamma:=\Gamma_\bfS=\Gamma_\bC$ and that $\v_\bfS$ is conjugate to $\v_\bC$ under the Weyl group of $\Gamma$ (see \cite[Lemma 5.3.9]{letellier4} for more details). Hence $\v_\bC$ is also a root.

When $g=0$, we already know by Theorem \ref{nemp} that if $\v_\bC$ is a root, then $\M_{\overline{\bC}}$ and so $\bM_{\bf L,P,\sigma}$ is not empty. If $g>0$ then $\v_\bC$ is always in the fundamental domain of imaginary root by \cite[Lemma 5.2.3]{HLV2} and so $\M_\bfS$ is not empty by \cite[Theorem 5.3.10]{HLV2}. Hence both $\bM_{\bf L,P,\sigma}$ and $\M_{\overline{\bC}}$ are not empty. 

\end{proof}

\subsection{Proof of Theorem \ref{counting}}\label{proofcounting}

\subsubsection{Frobenius formula for resolutions of character varieties}

To alleviate the notation, we put $G=\GL_n(\overline{\F}_q)$ and denote by $F:G\rightarrow G$ the Frobenius $(a_{ij})_{i,j}\mapsto (a_{ij}^q)_{i,j}$. For a subgroup $H$ of $G$ which is $F$-stable, we denote by $H^F$ the groups of fixed points. Recall that for any $F$-stable Levi decomposition $P=L\rtimes U_P$, with $P$ a parabolic subgroup of $G$ and $U_P$ the unipotent radical of $P$, we have an induction (called Harish-Chandra induction) $R_L^G: C(L^F)\rightarrow C(G^F)$ defined as 

$$
R_L^G(f)(g)=\frac{1}{|P^F|}\sum_{\{h\in G^F\,|\, h^{-1}gh\in P\}}f(\pi_P(h^{-1}gh))
$$
for any $f\in C(L^F)$ and where $\pi_P:P\rightarrow L$ is the canonical projection.

The following lemma is straightforward.

\begin{lemma}For any $F$-stable triple  $(\bf L,P,\sigma)$ as in \S \ref{resolchar}, we have 

$$
\# \bM_{\bf L,P,\sigma}(\F_q)=(q-1)\left\langle \calE*R_{L_1}^G(1_{\sigma_1})*\cdots *R_{L_k}^G(1_{\sigma_k}),1_1\right\rangle_{G^F},
$$
where $\calE$ is defined in \S \ref{frobsec}.
\label{frobprelimi}\end{lemma}

For a semisimple element $\sigma\in G$, denote by $G_\sigma$ the subset of elements of $G$ whose semisimple part is $G$-conjugate to $\sigma$. Note that the map which sends a unipotent conjugacy class  $C$ of the centralizer $G_G(\sigma)$ of $\sigma$ to the $G$-conjugacy classe of $\sigma\cdot C$ is a bijection onto the $G$-orbits of $G_\sigma$.

\begin{proposition} We have 

\beq
\# \bM_{\bf L,P,\sigma}(\F_q)=(q-1)\sum_{(C_1,\dots,C_k)}\left(\prod_{i=1}^k R_{L_i}^G(1_{\sigma_i})(C_i)\right)\sum_{\chi\in{\rm Irr}\,(G^F)}\left(\frac{|G^F|}{\chi(1)}\right)^{2g-2}\prod_{i=1}^k\frac{\chi(C_i)\, |C_i|}{\chi(1)},
\eeq
where the first sum is over the tuples of $(G^F)^k$-conjugacy classes of $G_{\sigma_1}^F\times\cdots\times G_{\sigma_k}^F$.

\end{proposition}

\begin{proof}This follows from Lemma \ref{frobprelimi} by writing formally $R_L^G(1_\sigma)$ as a sum of characteristic functions of conjugacy classes and then by using Theorem \ref{frob} together with Formula (\ref{conv}). The fact that the first sum is over conjugacy classes of $G_{\sigma_i}^F$ follows from the easy fact that the functions $R_{L_i}^G(1_{\sigma_i})$ are supported by $G_{\sigma_i}^F$.
\end{proof}

\subsubsection{Irreducible characters of $\GL_n(\F_q)$}

We use the same notation as in the previous section. Recall that for any $F$-stable Levi subgroup $L$ of some (non necessarily $F$-stable) parabolic subgroup $P$ of $G$, Deligne and Lusztig \cite{DLu} \cite{Lufinite} defined a $\kappa$-linear map $R_L^G:C(L^F)\rightarrow C(G^F)$, called \emph{Deligne-Lusztig induction}, which maps virtual characters (i.e. $\Z$-linear combination of characters) to virtual characters. It is well-known that this induction does not depend on the choice of $P$ and it is clear from the definition that it coincides with Harish-Chandra induction when the parabolic $P$ is $F$-stable. An irreducible character of $L^F$ is called \emph{unipotent} if it arises as a direct summand of some character $R_T^L({\rm Id})$ for some $F$-stable maximal torus of $L$.

It is well-known \cite{LSr} that any irreducible character of $G^F$ can be written in the form 

$$
R_{L,\calA^L,\theta}=\epsilon_G\epsilon_L R_L^G(\theta\cdot\calA^L)
$$
for some triple $(L,\calA^L,\theta)$ where $L$ is an $F$-stable Levi subgroup, $\calA^L$ a unipotent character of $L^F$ and $\theta:L^F\rightarrow\kappa^\times$ a  linear character satisfying the properties (a) and (b) of \cite[3.1]{LSr}. Also here $\epsilon_L$ denotes the sign $(-1)^{\F_q-\text{rank}\,(L)}$. By analogy with Jordan decomposition for conjugacy classes, $\theta$ and $\calA^L$ should be thought respectively as the semisimple and the unipotent parts of $R_{L,\calA^L,\theta}$.

\subsubsection{Proof of Theorem \ref{counting}}\label{proofcounting1}

We define the \emph{type} of $R_{L,\calA^L,\theta}$ as the  $G^F$-conjugacy class of the pair $(L,\calA^L)$.
For a pair $(L,\calA^L)$ we define $W_{G^F}(L,\calA^L)$ as the quotient by $L^F$ of the normalizer in $G^F$ of the pair $(L,\calA^L)$.
 
If $L$ is an $F$-stable Levi, then we have an isomorphism $L\simeq \prod_{i=1}^r(\GL_{n_i}(\overline{\F}_q))^{d_i}$ such that the Frobenius $F$ on $L$ corresponds to the Frobenius $vF$ where $v$ acts on each component $(\GL_{n_i})^{d_i}$ by circular permutation of the coordinates. Hence 

$$
L^F\simeq \prod_{i=1}^r\GL_{n_i}(\F_{q^{d_i}}).
$$
We then define the constant 

$$
K_L^o=\begin{cases}(-1)^{r-1}d^{r-1}\mu(d)(r-1)!&\text{ if for all } i, d_i=d,\\
0&\text{ if not.}\end{cases}
$$
Here $\mu$ denotes the classical M\"obius function.

\begin{theorem}For any $F$-stable generic tuple $(C_1,\dots,C_k)$ of conjugacy classes of $G$, we have 

\beq
\sum_\chi\prod_{i=1}^k\chi(C_i^F)=\frac{(q-1)K_L^o}{|W_{G^F}(L,\calA^L)|}\prod_{i=1}^k\epsilon_G\epsilon_L R_L^G(\calA^L)(C_i^F),
\eeq
where $\chi$ runs over the irreducible characters of $G^F$ of type $(L,\calA^L)$.
\label{thHLV}\end{theorem}

\begin{proof} This \cite[Theorem 4.3.1 (ii)]{HLV} re-written in the language of groups. 
\end{proof}

Recall that the  value of $R_{L,\calA^L,\theta}$ at $1$ does not depend on $\theta$.

\begin{corollary}For any $F$-stable $(\bf L,P,\sigma)$ as in \S \ref{resolchar} and any $F$-stable generic tuple $\bfS=(S_1,\dots,S_k)$ of conjugacy classes of $G$ we have

\begin{align*}
&\# \bM_{\bf L,P,\sigma}(\F_q)=\sum_{(M,\calA^M)}\frac{(q-1)^2K_M^o(\epsilon_G\epsilon_M)^k}{|W_{G^F}(M,\calA^M)|}\left(\frac{|G^F|}{R_{M,\calA^M,{\rm Id}}(1)}\right)^{2g-2+k}\prod_{i=1}^k\left\langle R_{L_i}^G(1_{\sigma_i}),R_M^G(\calA^M)\right\rangle_{G^F},\\
&\# \M_\bfS(\F_q)=\sum_{(M,\calA^M)}\frac{(q-1)^2K_M^o(\epsilon_G\epsilon_M)^k}{|W_{G^F}(M,\calA^M)|}\left(\frac{|G^F|}{R_{M,\calA^M,{\rm Id}}(1)}\right)^{2g-2}\prod_{i=1}^k\frac{|S_i^F|\cdot R_M^G(\calA^M)(S_i^F)}{R_{M,\calA^M,{\rm Id}}(1)}.
\end{align*}
where the sum runs over the $G^F$-conjugacy classes of pairs of the form $(M,\calA^M)$ with $M$ an $F$-stable Levi subgroup (of some parabolic subgroup of) $G$ and $\calA^M$ a unipotent character of $M^F$.

\label{frob-group}\end{corollary}

\begin{proof} It follows from Theorem \ref{thHLV} and the Frobenius formulas together with the fact that for any linear character $\theta:M^F\rightarrow\kappa^\times$, we have $R_{M,\calA^M,\theta}(1)=R_{M,\calA^M,{\rm Id}}(1)$.
\end{proof}

Hence to prove Theorem \ref{counting} we are reduced to prove the following theorem.

\begin{theorem}For any $F$-stable Levi subgroups $L$ and $M$, any elements $\sigma$ in $(Z_L)^F$ and any unipotent character $\calA^M$ of $M^F$, we have 

$$
\left\langle R_L^G(1_\sigma),R_M^G(\calA^M)\right\rangle_{G^F}=|L^F|^{-1}R_M^G(\calA^M)(S^F),
$$
where $S$ is an $F$-stable semisimple conjugacy class of $G$ with a representative whose centralizer in $G$ is $L$.
\label{countingbis}\end{theorem}

Denote by $C(G^F)_{\rm uni}$ the subspace of $C(G^F)$ generated by the the Deligne-Lusztig characters $R_T^G({\rm Id})$ where $T$ runs over the $F$-stable maximal tori of $G$. Since $R_M^G(\calA^M)$ belongs to $C(G^F)_{\rm uni}$, to prove the above theorem we are reduced to prove 
that

\beq
\left\langle R_L^G(1_\sigma),R_T^G({\rm Id})\right\rangle_{G^F}=|L^F|^{-1}R_T^G({\rm Id})(S^F),
\eeq
for any $F$-stable maximal torus $T$ of $G$.

The following result will allow us to make a further reduction.

\begin{proposition} The number $\left\langle R_L^G(1_\sigma),R_T^G({\rm Id})\right\rangle_{G^F}$ does not depend on $\sigma\in (Z_L)^F$, i.e.,

$$
\left\langle R_L^G(1_\sigma),R_T^G({\rm Id})\right\rangle_{G^F}=\left\langle R_L^G(1_1),R_T^G({\rm Id})\right\rangle_{G^F},
$$
for any $\sigma\in (Z_L)^F$.

\end{proposition}

\begin{proof}We pass to the Lie algebra to use a Fourier transform argument. Denote by $\mathfrak{t}$, $\mathfrak{l}$, $\mathfrak{g}$ the Lie algebras of $T$, $L$, $G$ respectively, and denote by $C(\mathfrak{g}^F)$ the $\kappa$-vector space of functions $\mathfrak{g}^F\rightarrow\kappa$ that are constant on adjoint orbits of $\mathfrak{g}^F$. Then we can define a Lie algebra analogue of Deligne-Lusztig induction $R_\mathfrak{l}^\mathfrak{g}:C(\mathfrak{l}^F)\rightarrow C(\mathfrak{g}^F)$ for any $F$-stable Levi subgroup $L$ as in \cite{letellier} (see also \cite[\S 6.5]{letellier4} for a self-contained review). Recall that if $x\in\mathfrak{g}^F$ has Jordan decomposition $x_s+x_n$, with $x_s$ semisimple and $x_n$ nilpotent, and if $g=g_sg_u$ is the Jordan decomposition of $g\in G^F$ such that $C_G(g_s)=C_G(x_s)$ and $x_n=g_u-1$, then

$$
R_T^G({\rm Id})(g)=R_{\mathfrak{t}}^\mathfrak{g}({\rm Id})(x),
$$
and so 

$$
\left\langle R_L^G(1_\sigma),R_T^G({\rm Id})\right\rangle_{G^F}=\left\langle R_{\mathfrak{l}}^{\mathfrak{g}}(1_\zeta),R_\mathfrak{t}^\mathfrak{g}({\rm Id})\right\rangle_{G^F},
$$
for any $\sigma\in (Z_L)^F$ and $\zeta\in\mathfrak{g}^F$ is such that $C_G(\zeta)=C_G(\sigma)$. By notation abuse, we still denote by $\langle\,,\,\rangle_{G^F}$ the form on $C(\mathfrak{g}^F)$ which is defined by 

$$
(f_1,f_2)\mapsto \frac{1}{|G^F|}\sum_{x\in\mathfrak{g}^F}f_1(x)\overline{f_2(x)}.
$$
We are now reduced to prove that $\left\langle R_{\mathfrak{l}}^{\mathfrak{g}}(1_\zeta),R_\mathfrak{t}^\mathfrak{g}({\rm Id})\right\rangle_{G^F}$ does not depend on $\zeta$ in the center of $\mathfrak{l}^F$.

Fix a non-trivial additive character $\psi:\F_q\rightarrow\kappa^\times$ and define the arithmetic Fourier transforms $\calF^\mathfrak{g}:C(\mathfrak{g}^F)\rightarrow C(\mathfrak{g}^F)$ by 

$$
\calF^\mathfrak{g}(f)(x)=\sum_{y\in\mathfrak{g}^F}\psi\left({\rm Trace}\,(xy)\right) f(y).
$$
for any $f\in C(\mathfrak{g}^F)$. Then it is well-known that $q^{-\frac{{\rm dim}\, G}{2}}\calF^\mathfrak{g}$ is an isometry for $\langle\,,\,\rangle_{G^F}$ (see \cite{letellier} and the references therein for the basic properties of Fouriers transforms). Moreover by the main result of \cite{letellier} we have the following commutation formula

$$
\calF^\mathfrak{g}\circ R_\mathfrak{l}^\mathfrak{g}=\epsilon_G\epsilon_L q^{\frac{{\rm dim}\, G-{\rm dim}\, L}{2}}R_\mathfrak{l}^\mathfrak{g}\circ\calF^\mathfrak{l}.
$$
Hence

$$
\left\langle R_{\mathfrak{l}}^{\mathfrak{g}}(1_\zeta),R_\mathfrak{t}^\mathfrak{g}({\rm Id})\right\rangle_{G^F}=\epsilon_L\epsilon_T q^{-\frac{{\rm dim}\,L+{\rm dim}\, T}{2}}\left\langle R_\mathfrak{l}^\mathfrak{g}\big(\calF^\mathfrak{l}(1_\zeta)\big),R_\mathfrak{t}^\mathfrak{g}(\calF^\mathfrak{t}\big({\rm Id})\big)\right\rangle_{G^F}.
$$
Notice that $\calF^\mathfrak{t}({\rm Id})=|\mathfrak{g}^F|\cdot 1_0$ and that $\calF^\mathfrak{l}(1_\sigma)$ is the linear character $\mathfrak{l}^F\rightarrow \kappa^\times$, $x\mapsto\psi({\rm Trace}(x\sigma))$. Denote by ${\rm Res}^\mathfrak{g}_{\rm nil}$ the linear map $C(\mathfrak{g}^F)\rightarrow C(\mathfrak{g}^F)$ that maps a function $f$ to the nilpotently supported function that takes the same values as $f$ on the nilpotent elements of $\mathfrak{g}^F$. Denote by ${\rm Id}_{\rm nil}$ the image of the identity function by ${\rm Res}^\mathfrak{g}_{\rm nil}$. Since $R_\mathfrak{t}^\mathfrak{g}(1_0)$ is supported on nilpotent elements, we have 

\begin{align*}
 \left\langle R_{\mathfrak{l}}^{\mathfrak{g}}(1_\zeta),R_\mathfrak{t}^\mathfrak{g}({\rm Id})\right\rangle_{G^F}&=\epsilon_L\epsilon_T q^{{\rm dim}\,G-\frac{{\rm dim}\,L+{\rm dim}\, T}{2}}\left\langle R_\mathfrak{l}^\mathfrak{g}\big(\calF^\mathfrak{l}(1_\zeta)\big),R_\mathfrak{t}^\mathfrak{g}(1_0)\big)\right\rangle_{G^F}\\
&=\epsilon_L\epsilon_T q^{{\rm dim}\,G-\frac{{\rm dim}\,L+{\rm dim}\, T}{2}}\left\langle {\rm Res}^\mathfrak{g}_{\rm nil}\circ R_\mathfrak{l}^\mathfrak{g}\big(\calF^\mathfrak{l}(1_\zeta)\big),R_\mathfrak{t}^\mathfrak{g}(1_0)\big)\right\rangle_{G^F}\\
&=\epsilon_L\epsilon_T q^{{\rm dim}\,G-\frac{{\rm dim}\,L+{\rm dim}\, T}{2}}\left\langle R_\mathfrak{l}^\mathfrak{g}\circ{\rm Res}^\mathfrak{l}_{\rm nil}\big(\calF^\mathfrak{l}(1_\zeta)\big),R_\mathfrak{t}^\mathfrak{g}(1_0)\big)\right\rangle_{G^F}\\
&=\epsilon_L\epsilon_T q^{{\rm dim}\,G-\frac{{\rm dim}\,L+{\rm dim}\, T}{2}}\left\langle R_\mathfrak{l}^\mathfrak{g}\big({\rm Id}_{\rm nil}\big),R_\mathfrak{t}^\mathfrak{g}(1_0)\big)\right\rangle_{G^F}.
\end{align*}
The third identity follows from the fact that Deligne-Lusztig induction commutes with restriction to nilpotent elements.
\end{proof}

We are now reduced to prove the following result.

\begin{proposition}We the notation of Theorem \ref{countingbis} we have 

$$
\left\langle R_L^G(1_1),R_T^G({\rm Id})\right\rangle_{G^F}=|L^F|^{-1} R_T^G({\rm Id})(S^F).
$$
\end{proposition}

\begin{proof}Denote by $W_L$ the Weyl group of $L$ with respect to a split $F$-stable maximal torus $T^L$ of $L$ (i.e. a maximal torus which is contained in some $F$-stable Borel subgroup of $L$). Recall that two elements $h$ and $h'$ of a group $H$ endowed with an automorphsim $f:H\rightarrow H$ are said to be $f$-\emph{conjugate} if there exists $g\in H$ such that $w=gw'f(g)^{-1}$. Then recall that the $L^F$-conjugacy classes of the $F$-stable maximal tori of $L$ are parametrized by the $F$-conjugacy classes of $W_L$ so that $T^L$ corresponds to the $F$-conjugacy class of $1\in W_L$ (see \cite{letellier4} for a review). For $w\in W_L$, we denote by $T^L_w$ the $F$-stable maximal torus of $L$ in the $L^F$-conjugacy class corresponding to the $F$-conjugacy class of $w$. We have (see for instance \cite[Lemme 7.3.1]{letellier})

$$
R_L^G(1_1)=\frac{1}{|L^F|_p\,|W_L|}\sum_{w\in W_L}\epsilon_L\epsilon_{T_w^L}R_{T^L_w}^G(1_1).
$$
Recall the orthogonality relations for Green functions (for any two $F$-stable maximal tori $T_1$ and $T_2$ of $G$):

$$
\left\langle R_{T_1}^G(1_1), R_{T_2}^G(1_1)\right\rangle_{G^F}=\frac{|N_{G^F}(T_1,T_2)|}{|T_1^F|\, |T_2^F|},
$$
where $N_{G^F}(T_1,T_2)$ denotes the set of $g\in G^F$ such that $gT_1g^{-1}=T_2$. We thus have
$$
\left\langle R_L^G(1_1),R_T^G({\rm Id})\right\rangle_{G^F}= \frac{1}{|L^F|_p\,|W_L|}\sum_{w\in W_L}\epsilon_L\epsilon_{T_w^L}\frac{|N_{G^F}(T,T^L_w)|}{|T^F|\, |(T^L_w)^F|}.
$$
If $N_{G^F}(T,T^L_w)\neq\emptyset$, then $|(T^L_w)^F|=|T^F|$, $\epsilon_{T^L_w}=\epsilon_T$ and $N_{G^F}(T,T^L_w)\simeq N_{G^F}(T)$, hence

$$
\left\langle R_L^G(1_1),R_T^G({\rm Id})\right\rangle_{G^F}= \frac{\epsilon_L\epsilon_T |N_{G^F}(T)|}{|T^F|^2\,|L^F|_p\,|W_L|}\,\#\{w\in W_L\,|\,N_{G^F}(T,T^L_w)\neq\emptyset\}.
$$
Denote by $W_G(T^L)$ the Weyl group of $G$ with respect to $T^L$, hence $W_L\subset W_G(T^L)$. Then the $G^F$-conjugacy class of $T$ corresponds to a unique $F$-stable conjugacy class of $W_G(T^L)$ with representative, say $v\in W_G(T^L)$. Then $T$ and $T^L_w$ are $G^F$-conjugate if and only if $w$ and $v$ are $F$-conjugate in $W_G(T^L)$.

Hence

$$
\left\langle R_L^G(1_1),R_T^G({\rm Id})\right\rangle_{G^F}= \frac{\epsilon_L\epsilon_T |W_{G^F}(T)|}{|T^F|\,|L^F|_p\,|W_L|}\,\#\{w\in W_L\,|\,v\, \text{ and }w \text{ are } F\text{-conjugate}\}.
$$
Now by  Deligne-Lusztig \cite[Corollary 7.2]{DLu}, we have 

$$
R_T^G({\rm Id})(S^F)= \frac{\epsilon_T\epsilon_L}{|T^F|\, |L^F|_p}\,\#\{g\in G^F\,|\, gTg^{-1}\subset L\},
$$
where $|L^F|_p$ denotes the $p$-part of $|L^F|$.
 We are reduced to prove that 
 
 \beq
|L^F|^{-1}\,  \#\{g\in G^F\,|\, gTg^{-1}\subset L\}=|W_L|^{-1}\,|W_{G^F}(T)|\,\#\{w\in W_L\,|\,v\, \text{ and }w \text{ are } F\text{-conjugate}\}.
\label{for}\eeq
We can write 

$$
\#\{g\in G^F\,|\, gTg^{-1}\subset L\}=\sum_{(w)}\#\,\{g\in G^F\,|\, gTg^{-1}\text{ is } L^F\text{-conjugate to }T^L_w\},
$$
where the sum runs over the set $(w)$ of $F$-conjugacy classes of $W_L$ such that $w$ is $F$-conjugate to $v$ in $W_G(T^L)$.

The group $L^F$ acts on $\{g\in G^F\,|\, gTg^{-1}\text{ is } L^F\text{-conjugate to }T^L_w\}$ on the right as $g\cdot l=l^{-1}g$ for any $l\in L^F$. The map

$$
N_{G^F}(T,T_w^L)\rightarrow \{g\in G^F\,|\, gTg^{-1}\text{ is } L^F\text{-conjugate to }T^L_w\}/L^F,
$$
that maps $g$ to $gL^F$ is clearly surjective with fibres of cardinality $|N_{L^F}(T^L_w)|$.
Hence 
\begin{align*}
|L^F|^{-1}\,  \#\{g\in G^F\,|\, gTg^{-1}\subset L\}&=|N_{G^F}(T)|\,\sum_{(w)}\frac{1}{|N_{L^F}(T_w^L)|}\\
&=|W_{G^F}(T)|\,\sum_{(w)}\frac{1}{|W_{L^F}(T_w^L)|}\\
&=|W_{G^F}(T)|\,\sum_{(w)}\frac{1}{|(W_L)^{wF}|},
\end{align*}
where $wF$ is the Frobenius on $W_L$ given by $(wF)(u)=wF(u)w^{-1}$. Noticing that $\frac{|W_L|}{|(W_L)^{wF}|}$ equals the cardinality of the $F$-conjugacy class of $w$ in $W_L$, we deduce Formula (\ref{for}).

\end{proof}

\section{Mixed Hodge polynomial of character varieties}

\subsection{Preliminaries}\label{symmetric}

\subsubsection{Exp and Log}

Let $\x_1,\x_2,\dots,\x_k$ be disjoints sets of infinitely many variables and let $\Lambda:=\Q(z,w)\otimes_\Z\Lambda(\x_1)\otimes_\Z\cdots\otimes_\Z\Lambda(\x_k)$ the ring of functions separately symmetric in each set $\x_1,\x_2,\dots,\x_k$ with coefficients in $\Q(z,w)$ where $z$ and $w$ are indeterminates.

\noindent For an integer $n>0$, consider the \emph{Adams operation}  $$\psi_n:\Lambda[[T]]\rightarrow\Lambda[[T]],\, f(\x_1,\dots,\x_k;q,T)\mapsto f(\x_1^n,\dots,\x_k^n;q^n,T^n)$$where we denote by $\x^d$ the set of variables $\{x_1^d,x_2^d,\dots\}$.

Define $\Psi:T\Lambda[[T]]\rightarrow T\Lambda[[T]]$ by $$\Psi(f)=\sum_{n\geq 1}\frac{\psi_n(f)}{n}.$$Its inverse is given by $$\Psi^{-1}(f)=\sum_{n\geq 1}\mu(n)\frac{\psi_n(f)}{n}$$where $\mu$ is the ordinary M\"obius function. 

We define $\Log:1+T\Lambda[[T]]\rightarrow T\Lambda[[T]]$ and its inverse $\Exp:T\Lambda[[T]]\rightarrow 1+\Lambda[[T]]$ as 

$$\Log(f)=\Psi^{-1}\left(\log(f)\right)$$and $$\Exp(f)=\exp\left(\Psi(f)\right).$$

\subsubsection{Types of conjugacy classes of $\GL_n(\F_q)$}\label{conj-type}

Extend the total ordering $\geq$ on $\calP$ to the set of pairs $(d,\lambda)\in\Z_{>0}\times(\calP\backslash\{0\})$, where $0$ is the unique partition of $0$, as follows. If $\lambda\neq\mu$ and $\lambda\geq\mu$, then $(d,\lambda)\geq(d',\mu)$, and $(d,\lambda)\geq(d',\lambda)$ if $d\geq d'$. We denote by $\bT$ the set of non-increasing sequences $(d_1,\lambda^1)(d_2,\lambda^2)\cdots(d_r,\lambda^r)$ of elements of $\Z_{>0}\times(\calP\backslash\{0\})$. The size of $\omega=(d_1,\omega^1)\cdots(d_r,\omega^r)\in\bT$ is defined as $|\omega|:=\sum_id_i\,|\omega^i|$. The integers $d_1,\dots,d_r$ are called the \emph{degree} of $\omega$. We denote by $\bT_n$ the subset of $\omhat\in\bT$ of size $n$. While $\tT_n$ parametrizes the types of the conjugacy classes of $\GL_n(\K)$ (see \S \ref{ne}), the set $\bT_n$ parameterizes the types of the conjugacy classes of $\GL_n(\F_q)$ as follows. Recall (see beginning of \S \ref{proofcounting}) that $F$-stable Levi subgroups $L$ defines multi-sets of pairs of integers $\{(d_i,n_i)\}_{i=1,\dots,r}$ such that

$$
L^F\simeq \GL_{n_1}(\F_{q^{d_1}})\times\cdots\times\GL_{n_r}(\F_{q^{d_r}}).
$$
The  $L^F$-unipotent conjugacy classes of $L^F$ correspond then, via the above isomorphism,  to the unipotent conjugacy classes of $\prod_{i=1}^k\GL_{n_i}(\F_{q^{d_i}})$ which are parametrized by the set of multi-partitions $\calP_{n_1}\times\cdots\times\calP_{n_r}$. Hence the set $\bT_n$ parametrizes the set of $G^F$-conjugacy classes of pairs $(L,C)$ with $L$ an $F$-stable Levi and $C$ a unipotent conjugacy class of $L^F$. The $G^F$-conjugacy class of  $(L,C)$ is now what we call the \emph{type} of the $G^F$-conjugacy classes of the elements $g\in G^F$ such that $C_G(g_s)=L$ and  $g_u\in C$ where $g_s$ and $g_u$ are respectively the semisimple and unipotent parts of $g$.

If the $\GL_n(\F_q)$-conjugacy class of $g\in\GL_n(\F_q)$ is of type $\omega=(d_1,\omega^1)\cdots(d_r,\omega^r)$, then the $\GL_n(\overline{\F}_q)$-conjugacy class of $g$ is of type

$$
\omega^o=\underbrace{\omega^1\cdots\omega^1}_{d_1}\underbrace{\omega^2\cdots\omega^2}_{d_2}\cdots\underbrace{\omega^r\cdots\omega^r}_{d_r}\in\tT$$

Denote by $m:\bT\rightarrow\tT$, $\omega\mapsto\omega^o$ the map we just defined.

\subsubsection{Definition of $\H_\omhat(z,w)$}\label{def-H}

Denote by $\oT$ (resp. $\otT$) the set of multi-types $(\omega_1,\dots,\omega_k)$ in $\bT^k$ (resp. in $(\tT){^k}$) of same size, i.e., $|\omega_1|=\cdots=|\omega_k|$. We now define a family  $\{\H_\omega(z,w)\}_{\omhat\in\oT}$ of rational functions in the variables $z,w$.  

 Consider \cite{HLV}

\beq
\label{cauchygk}
\Omega(z,w):= \sum_{{\lambda}\in {\cal P}}
{\cal H}_\lambda^g(z,w)\left(\prod_{i=1}^k \tilde{H}_\lambda(\x_i;z^2,w^2)\right)T^{|\lambda|}\in 1+T\Lambda[[T]]
\eeq
with
$$
{\cal H}_{\lambda}^g (z,w):=\prod_{s\in\lambda}
\frac{(z^{2a(s)+1}-w^{2l(s)+1})^{2g}} {(z^{2a(s)+2}-w^{2l(s)})(z^{2a(s)}-w^{2l(s)+2})}
$$
where the product is over all cells $s$ of $\lambda$ with $a(s)$ and $l(s)$ its arm and leg length, respectively (see \cite[\S 2.3.5]{HLV} for more details), and with

$$
\tilde{H}_\lambda(\x;q,t):=\sum_\mu \tilde{K}_{\mu\lambda}(q,t)s_\mu(\x)
$$ the \emph{modified Macdonald symmetric functions} as in \cite{garsia-haiman} (as usual $s_\mu(\x)$ denotes the Schur symmetric function \cite{macdonald}).

Given a family $\{A_\mu(\x; z,w)\}_\mu$ of symmetric functions indexed by partitions, we extend its definition  to types as follows. 

For $\omega=(d_1,\omega^1)\cdots(d_r,\omega^r)\in\bT$, put

$$
A_\omega(\x;z,w):=A_{\omega^1}(\x^{d_1};z^{d_1},w^{d_1})A_{\omega^2}(\x^{d_2};z^{d_2},w^{d_2})\cdots A_{\omega^r}(\x^{d_r};z^{d_r},w^{d_r})\in\Lambda(\x).
$$

\begin{remark} Applying this with $A_\mu$ the Schur symmetric function $s_\mu$, we have the following. If $\omega=(1,(n_1))(1,(n_2))\cdots(1,(n_r))$, then $s_\omega$ is the complete symmetric function $h_\lambda$ with $\lambda$ the partition $(n_1,\dots,n_r)$. If $\omega=(d_1,1)(d_2,1)\cdots (d_r,1)$, then $s_\omega$ is the power symmetric function $p_\lambda$ with $\lambda=(d_1,\dots,d_r)$.

\end{remark}

For  a multi-type $\omhat=(\omega_1,\dots,\omega_k)\in\oT$, put 
$$
s_\omhat:=s_{\omega_1}(\x_1)\cdots s_{\omega_k}(\x_k)\in\Lambda.
$$

The extended Hall pairing on $\Lambda$ is defined as $\langle\,,\,\rangle:=\prod_{i=1}^k\langle\,,\,\rangle_i$ where $\langle\,,\,\rangle_i$ is the Hall pairing on $\Lambda(\x_i)$ which makes the basis of Schur functions orthonormal.

Finally we denote by $\lambda'$ the dual of a partition $\lambda$. We define the dual $\omhat'$ of a multi-type $\omhat\in\oT$, by replacing any partition appearing in $\omhat$ by its dual.

For $\omhat=(\omega_1,\dots,\omega_k)\in\oT$, define 

$$
\H_\omhat(z,w):=(-1)^{r(\omhat)}(z^2-1)(1-w^2)\left\langle \Log\,\Omega(z,w),s_{\omhat'}\right\rangle
$$
where $r(\omhat):=k|\omhat|+\sum_{i,j}|\omega^j_i|$ and where $\left\langle \Log\,\Omega(z,w),s_{\omhat'}\right\rangle$ denotes the extended Hall pairing of $s_{\omhat'}$ with the coefficient of $\Log\,\Omega(z,w)$ in $T^{|\omhat|}$.

These functions satisfy the following properties \cite[\S 2.3.6]{HLV}.

\begin{lemma}The function $\Omega(z,w)$ (and therefore $\H_\omhat(z,w)$) is invariant both under changing $(z,w)$ to $(w,z)$ and under changing $(z,w)$ to $(-z,-w)$.
\label{propH}\end{lemma}

\begin{remark}1. The functions $\H_\omhat(z,w)$ were first considered in \cite[Formula (1.1.3)]{HLV} in the case where each coordinate $\omega_i$ of $\omhat$ is of the form $(1,(1^{n_1}))(1,(1^{n_2}))\cdots(1,(1^{n_r}))$. For a multi-partition $\muhat$, the function $\H_\muhat(z,w)$ introduced in \cite{HLV} corresponds then to the function $\H_\omhat(z,w)$ with  $\mu=(n_1,n_2,\dots,n_r)$ replaced by $\omega=(1,(1^{n_1}))(1,(1^{n_2}))\cdots(1,(1^{n_r}))$.

\noindent 2. The functions $\H_\omhat(z,w)$  can be computed recursively using the tables of $(q,t)$-Kostka polynomials $\tilde{K}_{\mu\lambda}(q,t)$ (see for instance \cite[\S 1.5.5]{HLV} for $\omhat=(\omega_1,\dots,\omega_k)$ with each $\omega_i$ of the form $(1,(1^{n_1}))\cdots(1,(1^{n_r}))$). 
\end{remark}

\subsection{The conjectures}

We assume  that $\bC$ is a generic tuple of conjugacy classes of $\GL_n(\C)$.

\begin{conjecture} The mixed Hodge polynomial $IH_c(\M_{\overline{\bC}}\,;\,x,y,t)$ is a polynomial in $xy$ and $t$, and it depends only on the type  of $\bC$ (and not on the eigenvalues).
\label{indep}\end{conjecture}

This conjecture is a natural extension of the semisimple case \cite[Conjecture 1.2.1 (ii)]{HLV}.

The map $m:\bT\rightarrow\tT$ defined in \S \ref{symmetric} has  a natural section $\iota:\tT\rightarrow\bT$ that maps $\omega^1\omega^2\cdots\omega^r$ to $(1,\omega^1)(1,\omega^2)\cdots(1,\omega^r)$.

\begin{conjecture} Assume that $\bC$ is of type $\omhat\in\otT$. We have 

$$
IH_c(\M_{\overline{\bC}}\,;\, q,t)=(t\sqrt q)^{d_\bC}\;
\H_{\iota^k(\omhat)}\left(-{\frac 1{\sqrt q},t\sqrt q }
\right).
$$
In particular we have 
$$
PP_c(\M_{\overline{\bC}}\,;\,q):=PP_c(\M_{\overline{\bC}}\,;\,\sqrt{q},\sqrt{q})= q^{d_\bC/2}\H_{\iota^k(\omhat)}(0,\sqrt q),
$$
where $PP_c(X;x,y):=\sum_{p,q}ih_c^{p,q;p+q}(X)x^py^q$ is the pure part of the mixed Hodge polynomial.
\label{mainconj}\end{conjecture}

In the semisimple case, this conjecture is \cite[Conjecture 1.2.1 (iii)(iv)]{HLV}.

\begin{remark}  Conjecture \ref{mainconj} together Conjecture \ref{indep}  give a complete description of the mixed Hodge numbers $ih_c^{p,q;k}(\M_{\overline{\bC}})$. Indeed the first assertion of Conjecture \ref{indep} says   that $ih_c^{p,q;k}(\M_{\overline{\bC}})=0$ unless $p=q$.
\end{remark}

Conjecture \ref{mainconj} together with Lemma \ref{propH} implies the following one.
 
 \begin{conjecture}[Curious Poincar\'e duality]
 We have 
 
 $$
 IH_c\left(\M_{\overline{\bC}}\,;\,\frac{1}{qt^2},t\right)=(qt)^{-d_\bC}IH_c(\M_{\overline{\bC}}\,;\,q,t).
 $$
 \end{conjecture}
  See \cite[Conjecture 1.2.2]{HLV} for the semisimple case.

\subsection{$E^{ic}$-polynomial of character varieties}
 One of the main result of this paper is the following theorem.
 
\begin{theorem} Conjecture \ref{mainconj} is true after the specialization $(q,t)\mapsto (q,-1)$, namely we have

$$
E^{ic}(\M_{\overline{\bC}}\,;\,q)=q^{d_\bC/2}\H_{\iota^k(\omhat)}\left(\frac 1{\sqrt q},\sqrt q \right).
$$

\label{maintheo1}\end{theorem}

If $\bC$ is a tuple of semisimple conjugacy classes, this theorem is \cite[Theorem 1.2.3]{HLV}.

\begin{corollary} The $E^{ic}$-polynomial of $\M_{\overline{\bC}}$ is \emph{palindromic}, namely

$$
E^{ic}(\M_{\overline{\bC}}\,;\,q)=q^{d_\bC}E^{ic}(\M_{\overline{\bC}}\,;\,q^{-1}).
$$
\end{corollary}

The rest of this section is devoted to the proof of Theorem \ref{maintheo1}. Note that the strategy goes along the same lines as for the proof of its "additive" version \cite[Corollary 7.3.5]{letellier4} except for the calculation.

\subsubsection{Intersection cohomology complex on $\calU_{\overline{\bC}}$}

Let $\bC=(C_1,\dots,C_k)$ be a generic tuple of conjugacy classes of $\GL_n(\K)$.

Put 

$$
\calY_{\overline{\bC}}:=(\GL_n)^{2g}\times\overline{C}_1\times\cdots\times\overline{C}_k.
$$
The following theorem is the key ingredient for the proof of Theorem \ref{maintheo1}.

\begin{theorem} If $i:\calU_{\overline{\bC}}\hookrightarrow\calY_{\overline{\bC}}$ denotes the inclusion, then $i^*\left(\IC {\calY_{\overline{\bC}}}\right)\simeq\IC {\calU_{\overline{\bC}}}$.
\label{restheo}\end{theorem}

We are going to prove it using the following general result \cite[Proposition 3.2.1]{letellier4}.

\begin{proposition} Let $X$ be an irreducible algebraic variety together with a decomposition $X=\bigcup_{\alpha\in I} X_\alpha$ where $I$ is a finite set and where the $X_\alpha$ are locally closed irreducible subvarieties. Assume given an irreducible  subvariety $Z$ of $X$ such that 

(i) if $Z_\alpha:=X_\alpha\cap Z$ is not empty, then it is equidimensional and  ${\rm codim}_X\,X_\alpha={\rm codim}_Z\, Z_\alpha$.

\noindent Assume moreover that there exists a Cartesian diagram 

$$\xymatrix{\tilde{X}\ar^f[rr]&&X\\
\tilde{Z}\ar[rr]^g\ar^{\tilde{i}}[u]&&Z\ar^i[u]}$$such that the conditions (ii) and (iii) below are satisfied.

(ii) $f$ and $g$ are semi-small resolutions of singularities.

(iii) The restriction of the sheaf $\mathcal{H}^i(f_*(\kappa))$ to $X_\alpha$ is a locally constant sheaf for all $i$.

\noindent Let $i:Z\hookrightarrow X$ denotes the inclusion, then $i^*(\IC X)=\IC Z$.
\label{proprest}\end{proposition}

Consider the natural stratification 

$$
\calY_{\overline{\bC}}=\coprod_{\bC'\unlhd \bC}\calY_{\bC'},
$$
where $\calY_\bC=(\GL_n)^{2g}\times C_1\times\cdots\times C_k$.

By Theorem \ref{geoth}, if not empty, $\calY_{\bC'}\cap\calU_{\overline{\bC}}=\calU_{\bC'}$ is irreducible of dimension $2g-n^2+1+\sum_i{\rm dim}\, C_i'$. Clearly the codimension of $\calU_{\bC'}$ in $\calU_{\overline{\bC}}$ equals the codimension of $\calY_{\bC'}$ in $\calY_{\overline{\bC}}$ and so the condition (i) in Proposition \ref{proprest} is statisfied. 

Consider resolutions $\bY_{L_i,P_i,\sigma_i}\rightarrow\overline{C}_i$ for each $i=1,\dots,k$ as in \S \ref{resolchar} and put 

$$
\bY_{\bf L,P,\sigma}:=(\GL_n)^{2g}\times\bY_{L_1,P_1,\sigma_1}\times\cdots\times\bY_{L_k,P_k,\sigma_k}.
$$
Consider the following Cartesian diagram

\beq\xymatrix{\bY_{\bf L,P,\sigma}\ar[rr]&&\calY_{\overline{\bC}}\\
\bbU_{\bf L,P,\sigma}\ar[rr]\ar[u]&&\calU_{\overline{\bC}}\ar[u]}\label{cartdiag}\eeq
where the vertical arrows are the canonical inclusions. The horizontal arrows are both resolutions (see \S \ref{resolchar}) and by a well-known result of Lusztig the top one is semi-small (see \cite[\S 4.3.4]{letellier4} for a review). Since the diagram is Cartesian and  ${\rm dim}\, \calY_{\overline{\bC}}-{\rm dim}\,\calY_{\overline{\bC}'}= {\rm dim}\, \calU_{\overline{\bC}}-{\rm dim}\,\calU_{\overline{\bC}'}$ for all $\bC'\unlhd\bC$ such that $\calU_{\overline{\bC}{'}}\neq\emptyset$, we deduce that the bottom vertical arrow is also semi-small.

The assertion (iii) is well-known (see \cite[Proposition 4.3.19]{letellier} for a review).

As a conclusion Proposition \ref{proprest} applies to our situation and Theorem \ref{restheo} follows.

\begin{corollary} If the conjugacy classes $C_1,\dots,C_k$ are regular (i.e., of dimension $n^2-n$) or semisimple then the character variety $\M_{\overline{\bC}}$ is rationally smooth. 
\label{rat-smooth}\end{corollary}

\begin{proof} This follows from Theorem \ref{restheo} and the well-known fact that Zariski closure of regular conjugacy classes are rationally smooth.
\end{proof}

\subsubsection{Characteristic functions of conjugacy classes}

Assume that $\K$ is an algebraic closure of a finite field $\F_q$ and consider the standard Frobenius $F:(a_{ij})_{i,j}\mapsto (a_{ij}^q)_{i,j}$ on $G=\GL_n(\K)$. To alleviate the notation we put $\bfX_{\overline{C}}:=\bfX_{\IC {\overline{C}}}$ for an $F$-stable conjugacy class $C$ of $G$. Note that if $C$ is semisimple, then $\overline{C}=C$ and so $\bfX_{\overline{C}}=1_C$.

 For a partition $\lambda$ of $n$, denote by  $C_\lambda$ the corresponding unipotent conjugacy class of $\GL_n$.

Recall \cite{Greenpoly} that for any two partitions $\lambda,\mu\in\calP_n$, we have

\begin{align*}
\bfX_{\overline{C}_\lambda}(C_\mu^F):&=\sum_i{\rm dim}\,\calH_{x_\mu}^{2i}(\IC {\overline{C}{_\lambda}})\, q^i\\
&=q^{-n(\lambda)}\tilde{K}_{\lambda\mu}(q),
\end{align*}
where $x_\mu\in C_\mu$ and where $n(\lambda)=\sum_{i>0}(i-1)\lambda_i$. More generally, for any two  conjugacy classes $C_\omega$ and $C_\tau$ of $\GL_n(\F_q)$ of type $\omega=(d_1,\omega^1)\cdots(d_r,\omega^r),\tau=(d_1,\tau^1)\cdots(d_r,\tau^r)\in \bT_n$ and such that $C_\tau\subset\overline{C}_\omega$ we have

\beq
\bfX_{\overline{C}_\omega}(C_\tau^F)=q^{-n(\omega)}\tilde{K}_{\omega\tau}(q),
\label{eval}\eeq
where $n(\omega):=\sum_{i=1}^rd_i n(\omega^i)$ and $\tilde{K}_{\omega\tau}(q)=\prod_{i=1}^r\tilde{K}_{\omega^i\tau^i}(q^{d_i})$.

\subsubsection{Formula for $E^{ic}(\M_{\overline{\bC}}\,;\,q)$}

Let $\bC=(C_1,\dots,C_k)$ be a generic tuple of conjugacy classes of $\GL_n(\C)$. As in \cite[Appendix 7.1]{HLV}, we can define a finitely generated ring extension $R$ of $\Z$ and a $k$-tuple $(\calC_1,\dots,\calC_k)$ of $R$-schemes such that $\calC_i$ is a spreading out of $C_i$ and for any ring homomorphism $\varphi: R\rightarrow\F_q$, the tuple $(\calC_1^\varphi(\overline{\F}_q),\dots,\calC_k^\varphi(\overline{\F}_q))$ is generic of same type as $(C_1,\dots,C_k)$. 

We want to prove the following formula.

\begin{theorem} For any ring homomorphism $\varphi:R\rightarrow\F_q$ we have 

$$
E^{ic}(\M_{\overline{\bC}}\,;\, q)=(q-1)\left\langle \calE*\bfX_{\calC_1^\varphi(\overline{\F}_q)}*\cdots*\bfX_{\calC_k^\varphi(\overline{\F}_q)}\right\rangle_{\GL_n(\F_q)},
$$
where $\calE$ is as in \S \ref{frobsec}.
\label{for32}\end{theorem}

The $R$-schemes $\mathfrak{M}_{\overline{\bC}}$, $\mathfrak{U}_{\overline{\bC}}$ defined from the $R$-schemes $(\calC_1,\dots,\calC_k)$ in the same way as $\M_{\overline{\bC}}$, $\calU_{\overline{\bC}}$ are defined from $(C_1,\dots,C_k)$ are  spreadings out of $\M_{\overline{\bC}}$ and $\calU_{\overline{\bC}}$ respectively. Recall that for $R$ sufficiently "large", the standard constructions of GIT are compatible with base change (see for instance \cite[Appendix B]{CBvdb}), i.e., our case, for any ring homomorphism $\varphi:R\rightarrow k$ into a field $k$, we have $\mathfrak{U}_{\overline{\bC}}^\varphi//\PGL_n\simeq \mathfrak{M}_{\overline{\bC}}^\varphi$. 

Fix an homomorphism $\varphi:R\rightarrow\F_q$. Since the quotient map $q:\mathfrak{U}_{\overline{\bC}}^\varphi(\overline{\F}_q)\rightarrow \mathfrak{M}_{\overline{\bC}}^\varphi(\overline{\F}_q)$ is a principal $\PGL_n$-bundle, we have $q^*\left(\IC {\mathfrak{M}_{\overline{\bC}}^\varphi}\right)\simeq \IC {\mathfrak{U}_{\overline{\bC}}^\varphi}$ and so 

$$
\sum_{x\in \mathfrak{M}_{\overline{\bC}}^\varphi(\F_q)}\bfX_{\mathfrak{M}_{\overline{\bC}}^\varphi(\overline{\F}_q)}(x)=\frac{1}{|\PGL_n(\F_q)|}\sum_{x\in \mathfrak{U}_{\overline{\bC}}^\varphi(\F_q)}\bfX_{\mathfrak{U}_{\overline{\bC}}^\varphi(\overline{\F}_q)}(x).
$$
From Theorem \ref{restheo}, we have 

$$
\bfX_{\mathfrak{U}_{\overline{\bC}}^\varphi(\overline{\F}_q)}(x)=\bfX_{\calC_1^\varphi(\overline{\F}_q)}(x_1)\cdots\bfX_{\calC_k^\varphi(\overline{\F}_q)}(x_k),
$$
for any $x=(a_1,b_1,\dots,a_g,b_g,x_1,\dots,x_k)\in\mathfrak{U}_{\overline{\bC}}^\varphi(\F_q)$.

 We deduce that 

$$
\sum_{x\in \mathfrak{M}_{\overline{\bC}}^\varphi(\F_q)}\bfX_{\mathfrak{M}_{\overline{\bC}}^\varphi(\overline{\F}_q)}(x)=(q-1)\left\langle \calE*\bfX_{\calC_1^\varphi(\overline{\F}_q)}*\cdots*\bfX_{\calC_k^\varphi(\overline{\F}_q)}\right\rangle_{\GL_n(\F_q)}.
$$
To prove Theorem \ref{for32} we are reduced to prove that 

$$
E^{ic}(\M_{\overline{\bC}}\,;\,q)=\sum_{x\in \mathfrak{M}_{\overline{\bC}}^\varphi(\F_q)}\bfX_{\mathfrak{M}_{\overline{\bC}}^\varphi(\overline{\F}_q)}(x).
$$
We need to verify that $\M_{\overline{\bC}}$ satisfies the property (E) above Theorem \ref{Katz2} with respect to the stratification $\coprod_{\bC'\unlhd\bC}\M_{\bC'}$. This is clear (from what we already saw) if we consider the natural morphism of $R$-schemes $\nabla:\mathfrak{M}_{\bf L,P,\sigma}\rightarrow\mathfrak{M}_{\overline{\bC}}$ which gives back the resolution $\nabla^\C:\bM_{\bf L,P,\sigma}\rightarrow\M_{\overline{\bC}}$ of \S \ref{resolchar} after extension of scalars from $R$ to $\C$. Indeed the condition (1)  of the property (E) is satisfied as the strata $\M_{\bC'}$ as well as $\M_{\bf L,P,\sigma}$ are polynomial count. For the condition  (3),  we are reduced, using Theorem \ref{restheo} and the Cartesian diagram (\ref{cartdiag}) to the case of the Zariski closures of  single conjugacy classes, which case is well-known (see \cite[\S 6.4]{letellier4} for a review). For the condition (2) we use the same argument as for (3) to get the analogue of Formula (\ref{isops}) for mixed Hodge modules (with $\nabla^\varphi$ replaced by $\nabla^\C$)  and then we take hypercohomology to get Formula (\ref{isomhs}) (we need to use \cite[Theorem 5]{Gottsche-Soergel} to apply the decomposition theorem for mixed Hodge modules). See also Theorem \ref{theoweight}  for the more general context with partial resolutions.

\subsubsection{Proof of Theorem \ref{maintheo1}}

We put $G=\GL_n(\overline{\F}_q)$ and denote by $F:G\rightarrow G$ the Frobenius $(a_{ij})_{i,j}\mapsto (a_{ij}^q)_{i,j}$.

Theorem \ref{maintheo1} follows from Theorem \ref{for32} together with Theorem \ref{convtheo} below. However note that Theorem \ref{convtheo} is more general that what we need as we are not assuming that the eigenvalues of conjugacy classes are in $\F_q$.

\begin{theorem} Assume given a generic tuple $\bC=(C_1,\dots,C_k)$ of $F$-stable conjugacy classes of $G$ such that $(C_1^F,\dots,C_k^F)$ is of type $\omhat\in\oT_n$.

Then 
$$
(q-1)\left\langle \calE*\bfX_{\overline{C}_1}*\cdots*\bfX_{\overline{C}_k},1_1\right\rangle_{G^F}=q^{d_\bC/2}\H_\omhat\left(\frac{1}{\sqrt{q}},\sqrt{q}\right).
$$
\label{convtheo}\end{theorem}

For an arbitrary tuple $\bC$, we decompose

\beq
\left\langle \calE*\bfX_{\overline{C}_1}*\cdots*\bfX_{\overline{C}_k},1_1\right\rangle_{G^F}=\sum_{\bfS}\left(\prod_{i=1}^k \bfX_{\overline{C}_i}(C_i'{^F})\right)\left\langle 1_{S_1^F}*\cdots* 1_{S_k^F},1_1\right\rangle_{G^F},
\label{decomp1}\eeq
where the sum is over the $F$-stable tuples $\bfS\unlhd\bC$. 

We saw in \S \ref{conj-type} that the $G^F$-conjugacy classes of pairs $(M,C)$ with $M$ an $F$-stable Levi and $C$ a unipotent conjugacy class of $M$ are parametrized by the set $\bT_n$. We have a natural bijection between the unipotent conjugacy classes of $M^F$ and the unipotent characters of $M^F$ so that the the identity character of $M^F$ corresponds to the unipotent regular conjugacy class. This bijection is given by looking at the support of the restrictions of  unipotent characters to unipotent elements (the support being the Zariski closure of a unipotent conjugacy class). Therefore we have also a natural parametrization of the $G^F$-pairs of the form $(M,\calA^M)$ by $\bT_n$.

Let us start with the following proposition.

\begin{proposition} For any generic tuple $\bfS=(S_1,\dots,S_k)$ of $F$-stable conjugacy classes of $G$ such that $(S_1^F,\dots,S_k^F)$ is of type $\tauhat=(\tau_1,\dots,\tau_k)\in(\bT_n)^k$, we have 

\beq
\left\langle 1_{S_1^F}*\cdots* 1_{S_k^F},1_1\right\rangle_{G^F}=\sum_{\alpha\in\bT_n}\frac{(q-1)K_\alpha^o}{|W(\alpha)|}\left(\frac{|G^F|}{\chi_\alpha(1)}\right)^{2g-2+k}(-1)^{kn+kf(\alpha)}\prod_{i=1}^k\frac{|S_i^F|\,\left\langle s_\alpha(\x_i),\tilde{H}_{\tau_i}(\x_i;q)\right\rangle}{|G^F|},
\label{DHLV}\eeq
where for $\alpha=(d_1,\alpha^1)\cdots(d_r,\alpha^r)\in\bT_n$ corresponding to the $G^F$-conjugacy class of $(M,\calA^M)$, we put $W(\alpha):=W_{G^F}(M,\calA^M)$, $K^o_\alpha:=K_M^o$, $\chi_\alpha:=R_{M,\calA^M,{\rm Id}}$ and $f(\alpha):=\sum_{i=1}^r|\alpha^i|$ so that $\epsilon_G\epsilon_M=(-1)^{n+f(\alpha)}$.
\end{proposition}

\begin{proof}This formula is exactly the second formula of Corollary \ref{frob-group} noticing that 

\beq
R_M^G(\calA^M)(S_i^F)=\left\langle s_\alpha(\x_i),\tilde{H}_{\tau_i}(\x_i;q)\right\rangle.
\label{uni-sym}\eeq
See also the calculation in \cite[Page 384]{HLV}.
\end{proof}

Now by \cite[Chapter VI, (6.7)]{macdonald}, we have 

\beq
\frac{|G^F|}{\chi_{\alpha}(1)}=(-1)^{f(\alpha)}H_\alpha(q) q^{\frac{1}{2}n(n-1)-n(\alpha)}.
\label{6.7}\eeq
where for a partition $\lambda$, we denote by $H_\lambda(q)$ the \emph{hook polynomial} $\prod_{s\in\lambda}(1-q^{h(s)})$ (see \cite[Chapter I, Part 3, Example 2]{macdonald}), and for a type $\alpha=(d_1,\alpha^1)\cdots(d_r,\alpha^r)$, we put $H_\alpha(q)=\prod_{i=1}^rH_{\alpha^i}(q^{d_i})$.

Also for a conjugacy class $C$ of $G^F$ of type $\tau$, we have 

$$
\frac{|C|}{|G^F|}=\calH_\tau(q),
$$
where for a partition $\lambda$ we put $\calH_\lambda(q):=\frac{1}{a_\lambda(q)}$ with $a_\lambda(q)$ the cardinality of the centralizer of a unipotent element of $G^F$ of type $\lambda$, and where $\calH_\tau(q)=\prod_{i=1}^r\calH_{\tau^i}(q^{d_i})$. 

Put $C_\alpha^o=K_\alpha^o/|W(\alpha)|$. Formula (\ref{DHLV2}) becomes 

\bigskip

$\left\langle 1_{S_1^F}*\cdots* 1_{S_k^F},1_1\right\rangle_{G^F}=$

\beq\sum_{\alpha\in\bT_n}(q-1)C_\alpha^oq^{\frac{1}{2}n(n-1)(2g-2+k)}\left(H_\alpha(q)q^{-n(\alpha)}\right)^{2g-2+k}(-1)^{kn}\prod_{i=1}^k\left\langle s_\alpha(\x_i),\calH_{\tau^i}(q)\, \tilde{H}_{\tau_i}(\x_i;q)\right\rangle,
\label{DHLV2}\eeq

For any two multi-types $\omhat=(d_1,\omega^1)\cdots(d_r,\omega^r)$ and $\tauhat=(e_1,\tau^1)\cdots(e_s,\tau^s)$, say that $\tauhat\unlhd\omhat$ if $r=s$, $d_i=e_i$ and $\tau^i\unlhd\omega^i$ for all $i=1,\dots,r$. 

Notice that if $\bC$ is of type $\omhat$, then the map which sends a tuple $\bfS=(S_1,\dots,S_k)$ of $F$-stable conjugacy classes to the type of $(S_1,\dots,S_k)$ restricts to a bijection between the set of multi-types $\tauhat\unlhd\omhat$ and the set of tuples of $F$-stable conjugacy classes $\bfS$ of $G$ such that $\bfS\unlhd\bC$.

Then plugging Formula (\ref{DHLV2}) into Formula (\ref{decomp1}) and using Formula (\ref{eval}), we find
\bigskip

$\left\langle \calE*\bfX_{\overline{C}_1}*\cdots*\bfX_{\overline{C}_k},1_1\right\rangle_{G^F}$
\begin{align*}
&=\sum_{\alpha\in\bT_n}(q-1)C_\alpha^oq^{\frac{1}{2}n(n-1)(2g-2+k)-\sum_{i=1}^kn(\omega_i)}\left(H_\alpha(q)q^{-n(\alpha)}\right)^{2g-2+k}\prod_{i=1}^k\left\langle s_\alpha(\x_i),\sum_{\tau\unlhd\omega_i}(-1)^n\tilde{K}_{\omega_i\tau}(q)\calH_{\tau^i}(q)\, \tilde{H}_{\tau_i}(\x_i;q)\right\rangle,\\
&=(q-1)q^{\frac{1}{2}(n^2(2g-2+k)-kn-2\sum_{i=1}^kn(\omega_i))}\\
&\left\langle\sum_{\alpha\in\bT}C_\alpha^oq^{(1-g)|\alpha|}\left(H_\alpha(q)q^{-n(\alpha)}\right)^{2g-2+k}\prod_{i=1}^k s_\alpha(\x_i),\prod_{i=1}^k\sum_{\tau\unlhd\omega_i}(-1)^n\tilde{K}_{\omega_i\tau}(q)\calH_\tau(q)\, \tilde{H}_\tau(\x_i;q)\right\rangle.
\end{align*}

If $\x=\{x_1,x_2,\dots\}$ and $\y=\{y_1,y_2,\dots\}$ are two sets of infinitely many variables, we denote by $\x\y$ the set of variables $\{x_iy_j\}_{i,j}$.

\begin{proposition}Let $\y=\{y_1,y_2,\dots\}$ be an infinite set of variables. With the specialization $y_i=q^{i-1}$ for all $i$ we have 

$$
s_{\omega'}(\x\y)=(-1)^{f(\omega)}\sum_{\tau\unlhd\omega}\calH_\tau(q)\tilde{K}_{\omega\tau}(q)\tilde{H}_\tau(\x;q),
$$
for any $\omega\in\bT_n$.
\label{propmagic}\end{proposition}

\begin{remark}This formula generalizes \cite[Lemma 2.3.6]{HLV} noticing that for $\omega=(1,(1^{n_1}))\cdots(1,(1^{n_r}))$ we have $s_{\omega'}=h_\mu$ with $\mu$ the partition $(n_1,\dots,n_r)$.

\end{remark}

\begin{proof} Let $\omega=(d_1,\omega^1)\cdots(d_r,\omega^r)\in\bT_n$. Since 

\begin{align*}
&s_{\omega'}(\x\y)=\prod_{i=1}^r s_{\omega^i{'}}(\x^{d_i}\y^{d_i}),\\
&(-1)^{f(\omega)}\sum_{\tau\unlhd\omega}\calH_\tau(q)\tilde{K}_{\omega\tau}(q)\tilde{H}_\tau(\x;q)=\prod_{i=1}^r(-1)^{|\omega^i|}\sum_{\tau\unlhd\omega^i}\calH_\tau(q^{d_i})\tilde{K}_{\omega^i\tau}(q^{d_i})\tilde{H}_\tau(\x^{d_i};q^{d_i}),
\end{align*}
we are reduced to prove the proposition for $\omega$ a partition, say a partition $\mu$ of $n$.

Recall that $s_\mu(\x)=\sum_\rho z_\rho^{-1}\chi^\mu_\rho p_\rho(\x)$ where $\chi^\mu_\rho$ denotes the value of the irreducible character $\chi^\mu$ of the symmetric group $\mathfrak{S}_n$ at an element of cycle-type $\rho$ and where $z_\rho$ denotes the cardinality of the centralizer of an element of cycle-type $\rho$ . Hence

\begin{align*}
s_{\mu'}(\x\y)&=\sum_\rho z_\rho^{-1}\chi^{\mu'}_\rho p_\rho(\x\y)\\
&=\sum_\rho z_\rho^{-1}\chi^{\mu'}_\rho p_\rho(\y) p_\rho(\x)\\
&=\sum_\rho z_\rho^{-1}\chi^{\mu'}_\rho p_\rho(\y)\sum_\nu\chi^\nu_\rho s_\nu(\x)\\
&=\sum_\nu\left(\sum_\rho z_\rho^{-1}\chi^\nu_\rho\chi^{\mu'}_\rho p_\rho(\y)\right)s_\nu(\x).
\end{align*}

On the other hand we have

$$
\sum_{\tau\unlhd\mu}\calH_\tau(q)\tilde{K}_{\mu\tau}(q)\tilde{H}_\tau(\x;q)=\sum_\nu\left(\sum_\tau\calH_\tau(q)\tilde{K}_{\mu\tau}(q)\tilde{K}_{\nu\tau}(q)\right) s_\nu(\x).
$$
We are reduced to prove the following identity.

\beq
\sum_\tau\calH_\tau(q)\tilde{K}_{\mu\tau}(q)\tilde{K}_{\nu\tau}(q)=(-1)^n\sum_\rho z_\rho^{-1}\chi^\nu_\rho\chi^{\mu'}_\rho p_\rho(\y).
\label{Lus}\eeq
As we can see from Formula (\ref{uni-sym}) with $M=G$, the polynomial $\tilde{K}_{\mu\tau}(q)$ coincides with the evaluation of the unipotent character $\calU^\mu$ of type $\mu$ at the unipotent element of type $\tau$. Hence the right hand-side of (\ref{Lus}) is the inner product  of the restrictions of the unipotent characters $\calU^\mu$ and $\calU^\nu$ at unipotent elements. This formula is well-known and is an easy consequence of the orthogonality relations for Green polynomials \cite[Chapter III, (7.10)]{macdonald}

$$
\sum_\lambda \calH_\lambda(q) Q^\lambda_\rho(q) Q^\lambda_\sigma(q)=\delta_{\rho\sigma} y_\rho(q),
$$
with $Q^\lambda_\rho(q)=\sum_\mu\chi^\mu_\rho\tilde{K}_{\mu\lambda}(q)$ and 

\begin{align*}y_\rho(q)&:=q^{-|\rho|}z_\rho\prod_i(1-q^{-\rho_i})^{-1}\\
&=z_\rho(-1)^{\ell(\rho)}p_\rho(\y),
\end{align*}
where $\ell(\rho)$ denotes the length of the partition $\rho$.
\end{proof}

We deduce that
\bigskip

$\left\langle \calE*\bfX_{\overline{C}_1}*\cdots*\bfX_{\overline{C}_k},1_1\right\rangle_{G^F}$
\begin{align*}
&=(q-1)q^{\frac{1}{2}(n^2(2g-2+k)-kn-2\sum_{i=1}^kn(\omega_i))}(-1)^{kn-\sum_{i=1}^kf(\omega_i)}\left\langle\sum_{\alpha\in\bT}C_\alpha^oq^{(1-g)|\alpha|}\left(H_\alpha(q)q^{-n(\alpha)}\right)^{2g-2+k}\prod_{i=1}^k s_\alpha(\x_i),\prod_{i=1}^k s_{\omega'}(\x_i\y)\right\rangle\\
&=(q-1)q^{\frac{1}{2}(n^2(2g-2+k)-kn-2\sum_{i=1}^kn(\omega_i))}(-1)^{r(\omhat)}\left\langle\sum_{\alpha\in\bT}C_\alpha^oq^{(1-g)|\alpha|}\left(H_\alpha(q)q^{-n(\alpha)}\right)^{2g-2+k}\prod_{i=1}^k s_\alpha(\x_i\y),\prod_{i=1}^k s_{\omega'}(\x_i)\right\rangle\\
&=(1-q^{-1})q^{\frac{1}{2}d_\bC}(-1)^{r(\omhat)}\left\langle\Log\,\Omega\left(\sqrt{q},\frac{1}{\sqrt{q}}\right),s_{\omhat'}\right\rangle\\
&=(q-1)^{-1}q^{\frac{1}{2}d_\bC}\H_\omhat\left(\sqrt{q},\frac{1}{\sqrt{q}}\right).
\end{align*}

For the second equality we use that for any two symmetric functions $u$ and $v$ we have $\langle u(\x\y),v(\x)\rangle=\langle u(\x),v(\x\y)\rangle$ (this can be checked in the basis of power symmetric functions). For the third equality we use that \cite[lemma 2.3.8]{HLV}

$$
\Omega(\sqrt{q},1/\sqrt{q})=\sum_\lambda q^{(1-g)|\lambda|}\big(q^{-n(\lambda)}H_\lambda(q)\big)^{2g+k-2}\prod_{i=1}^ks_\lambda(\x_i\y),
$$
and that for any family $\{A_\mu\}_\mu$ of symmetric functions indexed by partitions with $A_\emptyset=1$, we have \cite[Formula (2.3.9)]{HLV}

$$
\Log\,\left(\sum_\lambda A_\lambda T^{|\lambda|}\right)=\sum_\omega C_\omega^o A_\omega T^{|\omega|}.
$$

We thus proved Theorem \ref{convtheo}.

\section{Partial resolutions of character varieties and Weyl group action}\label{parresol}

When $\bC=(C_1,\dots,C_k)$ is a generic tuple of $F$-stable conjugacy classes of $G=\GL_n(\overline{\F}_q)$ with eigenvalues in $\F_q$, Theorem \ref{for32} gives an interpretation of the coefficients of the polynomials $\left\langle\calE*\bfX_{\overline{C}_1}*\cdots*\bfX_{\overline{C}_k}\right\rangle_{G^F}$ in terms of the mixed Hodge numbers of the corresponding complex character variety.  Our motivation for introducing partial resolutions of character varieties is to extend this theorem to the case where  the eigenvalues of $C_1,\dots,C_k$ are not necessarily in $\F_q$. The approach follows closely what we already did in the additive case \cite{letellier4} although the situation is slightly different as character varieties are not cohomologically pure unlike their additive version. We will therefore skip some details.

\subsection{Partial resolutions of conjugacy classes}\label{conj}

Let $P$ be a parabolic subgroup of $\GL_n(\K)$ and let $L$ be a Levi factor of $P$. Let $\Sigma$ be of the form $\Sigma=\sigma D$ where $\sigma$ lives in the center $Z_L$ of $L$ and $D$ is a unipotent conjugacy class of $L$. Denote by $U_P$ the unipotent radical of $P$. Put 

$$\bY_{L,P,\overline{\Sigma}}:=\left\{(x,gP)\in\GL_n\times (\GL_n/P)\,\left|\, g^{-1}x g\in\overline{\Sigma}. U_P\right\}\right.,$$and $\bY_{L,P,\Sigma}$ the open subset of elements $(x,gP)\in\bY_{L,P,\overline{\Sigma}}$ such that $g^{-1}xg\in\Sigma \cdot U_P$.  It is well-known that the image of $\bY_{L,P,\overline{\Sigma}}\rightarrow\GL_n$, $(x,gP)\rightarrow x$ is the Zariski closure $\overline{C}$ of some conjugacy class $C$ of $\GL_n$. The map $p:\bY_{L,P,\overline{\Sigma}}\rightarrow\overline{C}$ is a \emph{partial resolution} of $\overline{C}$. The variety $\bY_{L,P,\Sigma}$ is a dense non-singular irreducible open subset of $\bY_{L,P,\overline{\Sigma}}$. If $D=\{1\}$, then  $p$ is a resolution. 

Define 

$$
W_{\GL_n}(L,\Sigma):=\left\{h\in N_{\GL_n}(L)\,\left|\,h\Sigma h^{-1}=\Sigma\right\}\right./L.
$$
Note that $W_{\GL_n}(L,\Sigma)\subset W_{\GL_n}(L,D)$. The  group $W_{\GL_n}(L,\Sigma)$ acts on the perverse sheaf $p_*\left(\pIC {\bY_{L,P,\oSigma}}\right)$ (see \cite[\S 6.4]{letellier4} and the references therein for more details). More precisely we have 

\beq
p_*\left(\pIC {\bY_{L,P,\oSigma}}\right)=\bigoplus_{C'\unlhd \,C}A_{C'}\otimes \pIC {\overline{C}{'}},
\label{icconj}\eeq
where the direct sum is over the conjugacy classes contained in $\overline{C}$ and where $A_{C'}$ is some representation space (in general not irreducible) of $W_{\GL_n}(L,\Sigma)$. Although it depends on ${\bf L,\Sigma}$  we omitt ${\bf L,\Sigma}$ from the notation  $A_{C'}$ for simplicity  (see \cite[\S 6.4]{letellier4} for the explicit description  of the $W_{\GL_n}(L,\Sigma)$-modules $A_{C'}$). Taking the hypercohomology in (\ref{icconj}) we find that 

$$
IH_c^i(\bY_{L,P,\oSigma},\kappa)=\bigoplus_{C'\unlhd\, C}A_{C'}\otimes IH_c^{i+\delta_{C'}-\delta_C}(\overline{C}{'},\kappa),
$$
where $\delta_C$ means the dimension of $C$. The group $W_{\GL_n}(L,\Sigma)$ acts on  $IH_c^i(\bY_{L,P,\oSigma},\kappa)$ so that  for all $w\in W_{\GL_n}(L,\Sigma)$, we have

$$
{\rm Tr}\,\left(w,IH_c^i(\bY_{L,P,\oSigma},\kappa)\right)=\sum_{C'} {\rm Tr}\,\left(w, A_{C'}\right) {\rm dim}\, IH_c^{i+\delta_{C'}-\delta_C}(\overline{C}{'},\kappa).
$$

The $\GL_n(\K)$-conjugacy classes  of the pairs  $(L,D)$ as above are naturally parametrized by the elements of $\tT_n$. If $x=\sigma u$ is the Jordan decomposition of $x$ such that the centralizer of $\sigma$ in $\GL_n(\K)$  is $L$ and $u\in D$, then the type of the $\GL_n(\K)$-conjugacy class of $x$ (as defined in \S \ref{ne})  is the element of $\omega\in\tT_n$ corresponding to the conjugacy class of $(L,D)$ (we call also  $\omega$ the type of $(L,D)$).  Let $\omega\in\tT_n$ be the type of $(L,D)$.

Write $\omega\in\tT$ in the form 

$$
\omega=\underbrace{\alpha^1\cdots\alpha^1}_{a_1}\cdots\underbrace{\alpha^s\cdots\alpha^s}_{a_s}
$$
with $\alpha^i\neq\alpha^j$ if $i\neq j$.

Then $W_{\GL_n}(L,D)$ decomposes as a product of symmetric groups

$$
W_{\GL_n}(L,D)\simeq \prod_{i=1}^s\calS_{a_i}.
$$
Since the conjugacy classes of $\mathfrak{S}_n$ are parameterized by the partitions of $n$ we have a natural parameterization of the elements of the fiber $m^{-1}(\omega)\subset\bT$, where $m:\bT\rightarrow\tT$ is the map defined in \S \ref{symmetric},  by the conjugacy classes of $W_{\GL_n}(L,D)$.

Call the \emph{type} of the triple $(L,D,w)$, with $w\in W_{\GL_n}(L,D)$,  the element of $m^{-1}(\omega)$ corresponding to the conjugacy class of $w$.

\subsection{Partial resolutions of character varieties and Weyl group action}

We now consider $k$ triples $\{(L_i,P_i,\Sigma_i)\}_{i=1,\dots,k}$ as in  \S \ref{conj} with $\Sigma_i=\sigma_i\,D_i$ and denote by $\bbU_{\bfL,\bfP,\obfSigma}$ the subspace of elements $((a_j,b_j)_j,(x_i,g_iP_i)_i)$ of $\bY_{\bf L,P,\overline{\Sigma}}:=(\GL_n)^{2g}\times \prod_{i=1}^k\bY_{L_i,P_i,\overline{\Sigma}_i}$ which verify the equation

$$
(a_1,b_1)\cdots(a_g,b_g)x_1\cdots x_k=I,
$$
and denote by $\bbU_{\bfL,\bfP,\bfSigma}$ the subspace  $\bbU_{\bfL,\bfP,\obfSigma}\cap\left((\GL_n)^{2g}\times \prod_{i=1}^k\bY_{L_i,P_i,\Sigma_i}\right)$. For each $i=1,\dots,k$, the projection $p_i:\bY_{L_i,P_i,\oSigma_i}\rightarrow\GL_n$ on the first coordinate is the Zariski closure of a conjugacy class which we denote by $C_i$. The projection on the first $2g+k$ coordinates gives a surjection $p:\bbU_{\bfL,\bfP,\obfSigma}\rightarrow \calU_{\overline{\bC}}$ where $\bC=(C_1,\dots,C_k)$. The  diagonal action of $\GL_n$ on $\bbU_{\bfL,\bfP,\obfSigma}$ by conjugating the first $2g+k$ coordinates and by left multiplication on the last $k$ coordinates makes this map  $\GL_n$-equivariant.

We have the following theorem (whose proof is exactly the same as for Theorem \ref{nonsing}).

\begin{theorem}Assume that $\bC$ is generic and $\M_{\overline{\bC}}\neq\emptyset$. Then the geometric quotient $\bbU_{\bfL,\bfP,\overline{\bfSigma}}\rightarrow \bM_{\bfL,\bfP,\overline{\bfSigma}}$ exists and is a principal $\PGL_n$-bundle in the \'etale topology. Moreover it makes the following diagram Cartesian

$$
\xymatrix{\bbU_{\bfL,\bfP,\obfSigma}\ar[rr]^p\ar[d]&&\calU_{\overline{\bC}}\ar[d]\\
\bM_{\bfL,\bfP,\obfSigma}\ar[rr]^{p/_{\PGL_n}}&&\M_{\overline{\bC}}}
$$
\label{git}\end{theorem}

From now on we assume that $\bC$ is generic and that $\M_{\overline{\bC}}\neq\emptyset$. Let $\bM_{\bfL,\bfP,\bfSigma}$ be the image of $\bbU_{\bfL,\bfP,\bfSigma}$ in $\bM_{\bfL,\bfP,\obfSigma}$.

\begin{theorem}  $\bM_{\bfL,\bfP,\bfSigma}$ is an irreducible non-singular dense open subset of $\bM_{\bfL,\bfP,\obfSigma}$ of dimension $d_\bC$. Moreover, the restriction of $\IC {\bY_{\bf L,P,\overline{\Sigma}}}$ to $\bbU_{\bf L,P,\overline{\Sigma}}$ is isomorphic to $\IC {\bbU_{\bf L,P,\overline{\Sigma}}}$.\label{irrchar}\end{theorem}

\begin{proof} Consider the resolution $\hat{\pi}:\bM_{\bf \hat{L},\hat{P},\sigma}\rightarrow \bM_{\bf L,P,\overline{\Sigma}}$, $(X,g\hat{P})\mapsto (X,gP)$ where $\hat{L}_i$ and $\hat{P}_i$ are defined from $(L_i,P_i,D_i)$ as follows. We choose a parabolic $\tilde{P}_i$ of $L_i$ and a Levi factor $\hat{L}_i$ of $\tilde{P}_i$ such that the projection $\bY_{\hat{L}_i, \tilde{P}_i,\sigma_i}\rightarrow L_i$, $(x,l\tilde{P}_i)\mapsto x$ is a resolution of $\overline{\sigma_iD}_i$. Put  $\hat{P}_i:=\tilde{P}\cdot U_P$. It is the unique parabolic subgroup of $\GL_n$ having $\hat{L}_i$ as a Levi factor and contained in $P$. Then the map $\bM_{\bf \hat{L},\hat{P},\sigma}\rightarrow \bM_{\bf L,P,\overline{\Sigma}}$, $(xg\hat{P})\mapsto (x,gP)$ is surjective. By Corollary \ref{maincoro}, the variety $\bM_{\bf \hat{L},\hat{P},\sigma}$ is irreducible from which we get that $\bM_{\bf L,P,\overline{\Sigma}}$ is also irreducible. The proof of the second assertion goes exactly along the same lines as for the proof of Theorem \ref{restheo} using the resolution $\bM_{\bf \hat{L},\hat{P},\sigma}\rightarrow \bM_{\bf L,P,\overline{\Sigma}}$.
\end{proof}

Recall the Cartesian diagram

\beq
\xymatrix{\bY_{\bf L,P,\overline{\Sigma}}\ar[rr]^p&&\calY_{\overline{\bC}}\\
\bbU_{\bfL,\bfP,\obfSigma}\ar[u]\ar[rr]^p&&\calU_{\overline{\bC}}\ar[u]}
\label{cart}\eeq
where the vertical arrows are inclusions and the horizontal arrows projections on the first $2g+k$ coordinates, is Cartesian.  

Applying Theorem \ref{irrchar}, Formula (\ref{icconj}) and the proper base change theorem to the diagram (\ref{cart}) we find the following proposition (taking into account Tate twist).

\begin{proposition}If $\K=\C$ or if $\K=\overline{\F}_q$ and $({\bf L,P,\Sigma})$ is defined over $\F_q$, then 
$$
p_*\left(\pIC {\bbU_{\bfL,\bfP,\obfSigma}}\right)= \bigoplus_{\bC'\unlhd\bC}A_{\bC'}\otimes \pIC {\calU_{\overline{\bC}{'}}}(r_{\bC'}).
$$
where $A_{\bC'}=A_{C_1'}\otimes\cdots\otimes A_{C_k'}$ if $\bC'=(C_1',\dots,C_k')$, $r_{\bC'}=(d_{\bC'}-d_\bC)/2$ and where we use the same notation $\pIC {}$ for its mixed Hodge module version if $\K=\C$.
\end{proposition}

  Using that the diagram in Theorem \ref{git} is Cartesian we deduce (see \cite[Proposition 7.1.1]{letellier4} for more details) the following theorem.

\begin{theorem}Assume  that  $\K=\C$ or that $\K=\overline{\F}_q$ and $({\bf L,P,\Sigma})$ is defined over $\F_q$. Then 
\beq
(p/_{\PGL_n})_*\left(\pIC {\bM_{\bfL,\bfP,\obfSigma}}\right)= \bigoplus_{\bC'\unlhd\bC}A_{\bC'}\otimes \pIC {\M_{\overline{\bC}{'}}}(r_{\bC'}).
\label{decomp}\eeq
\label{theoweight}\end{theorem}

When $\K=\C$, taking then the hypercohomology in (\ref{decomp}) gives an isomorphism of mixed Hodge structures

\beq
IH_c^i\left(\bM_{\bfL,\bfP,\obfSigma},\C\right)\simeq \bigoplus _{\bC'\unlhd\bC}A_{\bC'}\otimes\left(IH_c^{i+2r_{\bC'}}(\M_{\overline{\bC}{'}},\C)\otimes \C(r_{\bC{'}})\right).
\label{decompHodge}\eeq

At the level of mixed Poincar\'e polynomials we deduce

\beq
IH_c\left(\bM_{\bfL,\bfP,\obfSigma};q,t\right)=\sum_{\bC'\unlhd\,\bC}\left({\rm dim}\,A_{\bC'}\right)\,(qt^2)^{-r_{\bC'}}\, IH_c\left(\M_{\overline{\bC}{'}};q,t\right).
\label{decomppoly}\eeq

Put

$$
W_{\GL_n}(\bfL,\bfSigma):=\prod_{i=1}^kW_{\GL_n}(L_i,\Sigma_i).
$$
Then $W_{\GL_n}(\bfL,\bfSigma)$ acts on $A_{\bC'}$ for all $\bC'\unlhd \bC$. We thus  get an action of $W_{\GL_n}(\bfL,\bfSigma)$ on $IH_c^k\left(\bM_{\bfL,\bfP,\obfSigma},\kappa\right)$ which preserves the weight filtration. Define the $\w$-twisted mixed Poincar\'e polynomial as
$$
IH_c^\w\left(\bM_{\bfL,\bfP,\obfSigma};q,t\right):=\sum_{m,i}{\rm Tr}\,\left(\w,{\rm Gr}_mIH_c^i(\bM_{\bf L,P,\obfSigma})\right)q^{m/2}t^k.
$$

We extend the definition of type of triple $(L,D,w)$ in \S \ref{conj} to multi-triples  $(\bfL,\bfD,\w):=(L_i,D_i,W_i)_i$ in the obvious way (the type of $(\bfL,\bfD,\w)$ is now an element in $\oT$). Note also that $W_{\GL_n}(\bfL,\bfSigma)\subset W_{\GL_n}(\bfL,\bfD)$.

\begin{conjecture} For any $\w\in W_{\GL_n}(\bfL,\bfSigma)$ we have

$$
IH_c^\w(\bM_{\bfL,\bfP,\obfSigma}; q,t)=(t\sqrt q)^{d_\bC}\;
\H_\omhat\left(-{\frac 1{\sqrt q},t\sqrt q }
\right)
$$
where $\omhat\in\oT$ is the type of $(\bfL,\bfD,\w)$. In particular we have 
$$
PP_c^\w(\bM_{\bfL,\bfP,\obfSigma};q)= q^{d_\bC/2}\H_\omhat(0,\sqrt q),
$$
where  $PP_c^\w(\bM_{\bfL,\bfP,\obfSigma};q):=\sum_i{\rm Tr}\,\left(\w,{\rm Gr}_{2i}IH_c^{2i}(\bM_{\bfL,\bfP,\obfSigma})\right)q^i$.
\label{main2}\end{conjecture}

\begin{remark}If $\sigma_i\in Z_{\GL_n}$ for all $i=1,\dots,k$, then $W_{\GL_n}(\bfL,\bfSigma)=W_{\GL_n}(\bfL,\bfD)$, and so the above conjecture gives a cohomological interpretation of the functions $\H_\omhat(z,w)$ for all $\omhat\in\bT$.
\end{remark}

\begin{proposition} Conjecture \ref{main2}  is equivalent to Conjecture \ref{mainconj}.

\label{impli}\end{proposition}

\begin{proof}It is clear that Conjecture \ref{main2} implies Conjecture \ref{mainconj} by taking $\w=1$ and, for each $i=1,\dots,k$,  by taking $\sigma_i$ such that $C_{\GL_n}(\sigma_i)=L_i$ as in this case the map $\bM_{\bf L,P,\overline{\Sigma}}\rightarrow\M_{\overline{\bC}}$ is an isomorphism. Let us prove the converse. By Formula (\ref{decomp}) we have

$$
IH_c^\w\left(\bM_{\bfL,\bfP,\obfSigma};q,t\right)=\sum_{\bC'\unlhd\,\bC}{\rm Tr}\left(\w\,|\, A_{\bC'}\right)\,(qt^2)^{-r_{\bC'}}\, IH_c\left(\M_{\overline{\bC}{'}};q,t\right), 
$$

Let $\omhat\in\oT$ be the type of $({\bf L,D},\w)$, see \S \ref{parresol}, and let $\omhat^+\in \otT$  be the type of $\bC$. Re-writing the above formula in terms of types  and assuming that Conjecture \ref{mainconj} is true we find

$$
(t\sqrt{q})^{-d_\bC}IH_c^\w\left(\bM_{\bfL,\bfP,\obfSigma};q,t\right)=\sum_{\tauhat\unlhd\omhat^+}{\rm Tr}\left(\w\,|\, A_\tauhat\right)\, \H_{\iota^k(\tauhat)}\left(-1/\sqrt{q},t\sqrt{q}\right).
$$

Proposition \ref{impli} is thus a consequence of the following identity between symmetric functions \cite[Formula (7.4.3)]{letellier4}

$$
(-1)^{r(\omhat)}s_{\omhat'}=\sum_{\tauhat\unlhd\omhat^+}{\rm Tr}\left(\w\,|\, A_\tauhat\right)\,s_{\tauhat'}.
$$

\end{proof}

Put 

$$
E^{ic}_\w(\bM_{\bfL,\bfP,\obfSigma};q):=IH_c^\w(\bM_{\bfL,\bfP,\obfSigma}; q,-1).
$$

Theorem \ref{maintheo1} implies the following one (see proof of Proposition \ref{impli}).

\begin{theorem} We have 

$$
E^{ic}_\w(\bM_{\bfL,\bfP,\obfSigma};q)=q^{d_\bC/2}\;
\H_\omhat\left({\frac 1{\sqrt q},\sqrt q }
\right).$$In other words Conjecture \ref{main2} is true after the specialization $(q,t)\mapsto (q,-1)$.
\label{maintheo2}\end{theorem}

We can always choose from the pair $({\bf L,D})$ a generic $k$-tuple ${\bf S}=(S_1,\dots,S_k)$ of conjugacy classes of $\GL_n(\K)$ such that for all $i=1,\dots,k$,  there exists an element of $S_i$ with Jordan decomposition $s_id_i$ where $C_{\GL_n}(s_i)=L_i$ and $d_i\in D_i$.

\begin{conjecture} We have

$$
IH_c\left(\bM_{\bfL,\bfP,\obfSigma};q,t\right)=IH_c\left(\M_{\overline{{\bf S}}};q,t\right).
$$
\label{main3}\end{conjecture}

This conjecture is an easy consequence of Conjecture \ref{mainconj} and Conjecture \ref{main2}. By Proposition \ref{impli} we get that Conjecture \ref{main3} is a consequence of Conjecture \ref{mainconj}.

By Theorem \ref{maintheo1} and Theorem \ref{maintheo2} we have the following generalization of Theorem \ref{E-poly-coro}.

\begin{theorem} We have 

$$
E^{ic}\left(\bM_{\bfL,\bfP,\obfSigma};q,t\right)=E^{ic}\left(\M_{\overline{{\bf S}}};q,t\right).
$$
\label{theoE}\end{theorem}

\section{Regular unipotent character varieties}\label{unipotent}

Denote by $\overline{\calP}$ the subset of multi-partitions $(\mu^1,\dots,\mu^k)\in(\calP) ^k$ with $|\mu^1|=|\mu^2|=\cdots=|\mu^k|$. 

Recall that $\{\H_\muhat(z,w)\}_{\muhat\in\oP}$ is defined by

$$
(z^2-1)(1-w^2)\,\Log\,\Omega(z,w)=\sum_{\muhat\in\overline{\calP}\backslash\{\emptyset\}}\H_\muhat(z,w)s_{\muhat'}.
$$
Put

$$
d_\muhat:= (2g+k-2)\,|\muhat|^2+2-\sum_{i,j}(n^i_j)^2,
$$
where $n^i_1,n^i_2,\dots$ are the parts of the dual partition of $\mu^i$.
Note that $d_\muhat=d_\bC$ for a tuple $\bC$ of conjugacy classes of type $\muhat$. Note also that $d_{((1^n),\dots,(1^n))}=(2g-2)n^2+2$ and $d_{((n^1),\dots,(n^1))}=(2g+k-2)n^2+2-kn$.

For each $\muhat\in\oP\backslash\{\emptyset\}$, choose a generic tuple $\bC$ of conjugacy classes of $\GL_{|\muhat|}(\K)$ type $\muhat$ and denote by $\M_\muhat^{\rm gen}/_\K$ the character variety $\M_{\overline{\bC}}$. 

\begin{theorem}For any finite field $\F_q$ we have

\begin{align}
(q-1)\,\Log\left(\sum_{\muhat\in\overline{\calP}}q^{1-d_\muhat/2}\left\langle\calE*\bfX_{\overline{C}_{\mu^1}}*\cdots*\bfX_{\overline{C}_{\mu^k}},1_1\right\rangle_{\GL_{|\muhat|}}\, s_{\muhat'}\right)&=q\sum_{\muhat\in\oP\backslash\{\emptyset\}}\H_\muhat\left(\sqrt{q},\frac{1}{\sqrt{q}}\right)\, s_{\muhat'}\label{conv2}\\
&=\sum_{\muhat\in\oP\backslash\{\emptyset\}}q^{1-d_\muhat/2}E^{ic}(\M_\muhat^{\rm gen}/_\C\,;\,q)\, s_{\muhat'}
\end{align}
where $(C_{\mu^1},\dots,C_{\mu^k})$ is a $k$-tuple of unipotent conjugacy classes of $\GL_{|\muhat|}(\F_q)$ of type $\muhat$.
\label{unichar-theo}\end{theorem}

The second identity of the theorem is  Theorem \ref{maintheo1}.

For any tuple $\bC$ (generic or not) of conjugacy classes of $\GL_n$, denote by $[\M_{\overline{\bC}}]$ the stack quotient of $\calU_{\overline{\bC}}$ by $\GL_n$.

Since Zariski closures of regular conjugacy classes are rationally smooth we have the following lemma.

\begin{lemma} Assume that $\K=\overline{\F}_q$ and $\bC=(C_1,\dots,C_k)$ is defined over $\F_q$. If  $C_1,\dots,C_k$ are either  semisimple or regular, then

\begin{align*}
\#[\M_{\overline{\bC}}](\F_q)&=\frac{\#\calU_{\overline{\bC}}(\F_q)}{|\GL_n(\F_q)|}\\
&=\left\langle\calE*\bfX_{\overline{C}_1}*\cdots*\bfX_{\overline{C}_k},1_1\right\rangle_{\GL_n(\F_q)}.
\end{align*}

\end{lemma}

\begin{remark} If we specialize the variables $\x_i=\{x_{i,1},x_{i,2},\dots\}$ to $\{T,0,0,\dots\}$ for all $i=1,\dots,k$ in the formulas of Theorem \ref{unichar-theo} we recover formula \cite[Theorem 3.8.1]{hausel-villegas}

\begin{align*}
\Log\left(\sum_{n\geq 0}\frac{\#\,{\rm Hom}\,(\pi_1(\Sigma),\GL_n(\F_q))}{q^{(g-1)n^2}|\GL_n(\F_q)|} \, T^n\right)&=(q-1)\, \Log\,\left(\sum_{\lambda\in\calP}\calH_\lambda^g(\sqrt{q},1/\sqrt{q})\, T^{|\lambda|}\right)\\
&=\frac{1}{1-q^{-1}}\sum_{n\geq 1}\H_{((1^n),\dots,(1^n))}\left(\sqrt{q},\frac{1}{\sqrt{q}}\right)\, T^n.
\end{align*}
\end{remark}
In the following corollary, $[\M_n^{\rm uni}]$ denotes the stack quotient $[\M_{\overline{\bC}}]$ with $\bC$ the unique $k$-tuple of unipotent regular conjugacy classes of $\GL_n(\F_q)$, and by $\M_n^{\rm gen}$ the generic character variety obtained from  a generic $k$-tuple of conjugacy classes defined over $\F_q$ of type $((n^1),\dots,(n^1))$.

\begin{corollary} We have 

\begin{align}
\Log\,\left(\sum_{n\geq 0}(-1)^{kn}q^{1-d_n/2}\#[\M_n^{\rm uni}](\F_q)\, T^n\right)&=\frac{1}{1-q^{-1}}\sum_{n\geq 1}(-1)^{kn}\H_{((n^1),\dots,(n^1))}\left(\sqrt{q},\frac{1}{\sqrt{q}}\right)\, T^n\label{reg}\\
&=\sum_{n\geq 1}(-1)^{kn}q^{1-d_n/2}\#[\M_n^{\rm gen}](\F_q)\, T^n,
\end{align}
where $d_n:=d_{(n^1),\dots,(n^1)}$.
\end{corollary}

The second identity in the corollary follows from the fact that $\M_n^{\rm gen}$ is rationally smooth (see Corollary \ref{rat-smooth}). The corollary gives thus an interesting relation between the stack counts of regular unipotent character varieties and the stack count of the corresponding generic character varieties of unipotent regular type (i.e. with regular conjugacy classes with only one eigenvalue).

\begin{proof} Consider the ring homomorphism $\Lambda(\x)\rightarrow\Q[u]$ which maps power sums $p_r(\x)$ to $1-u^r$. We define the \emph{$u$-specialization} of a symmetric function as its image under this ring homomorphism (in the plethystic notation this is $f[1-u]$). As in \cite[\S 3.1]{HLV3}, we define the top degree of a symmetric function $f\in\Lambda(\x)$ of homogeneous degree $n$ as

$$
[f]:=u^nf[1-u^{-1}]|_{u=0}.
$$
The operation $f\mapsto[f]$ commutes with $\Log$ (see \cite[Proposition 3.1]{HLV3}) and for a symmetric function $f$ of homogeneous degree $n$, we have \cite[Lemma 3.3]{HLV3}

$$
[f]=(-1)^n\langle f, s_{(1^n)}\rangle.
$$

We obtain Formula (\ref{reg}) by taking the top degree of $u$-specialization in Formula (\ref{conv2}).

\end{proof}

\begin{proof}[Proof of Formula (\ref{conv2})] The beginning of the proof follows closely that of Theorem \ref{convtheo}.

For a partition $\mu$ we have 

$$
\bfX_{\overline{C}_\mu}=\sum_{\lambda\unlhd\mu}q^{-n(\mu)}\tilde{K}_{\mu\lambda}(q) \, 1_{C_\lambda^F}.
$$Applying this formula together with Formulas (\ref{conv}) and (\ref{frobenius}) we find

$$
\left\langle\calE*\bfX_{\overline{C}_{\mu^1}}*\cdots*\bfX_{\overline{C}_{\mu^k}},1_1\right\rangle_{\GL_{|\muhat|}}=\sum_\chi\left(\frac{|\GL_{|\muhat|}|}{\chi(1)}\right)^{2g+k-1}\sum_{\lambdahat\unlhd\muhat}\prod_{i=1}^kq^{-n(\mu^i)}\tilde{K}_{\mu^i\lambda^i}(q)\calH_{\lambda^i}(q)\chi(C_{\lambda^i}^F).
$$
Denote by $\calO$ the set of Frobenius orbits of $\mathbb{G}_m$.  Recall that the irreducible characters of $\GL_n(\F_q)$ are parametrized by the set of all maps $h:\calO\rightarrow\calP$ such that 

$$
\sum_{\gamma\in\calO}|\gamma|\cdot|h(\gamma)|=n.
$$
This parametrization is chosen so that the function $h:\calO\rightarrow\calP$ that maps $\{1\}$ to the partition $\lambda\in\calP_n$ corresponds to the unipotent character of type $\lambda$. The type of the irreducible character $\chi_h$ associated with a function $h:\calO\rightarrow\calP$ is the type $\omega_h\in\bT$ given by the collection of pairs $(d,\lambda)\in\N\times(\calP\backslash\{\emptyset\})$ with multiplicity $\#\{\gamma\in\calO\,|\, (d,\lambda)=(|\gamma|,h(\gamma))\}$.

With this parametrization we have (see Formula (\ref{uni-sym}))

$$
\chi_h(C_\lambda^F)=(-1)^{|\omega_h|-f(\omega_h)}\left\langle\tilde{H}_\lambda(\x;q),s_{\omega_h}\right\rangle.
$$
Using Formula (\ref{6.7}) we have
\bigskip

$\left\langle\calE*\bfX_{\overline{C}_{\mu^1}}*\cdots*\bfX_{\overline{C}_{\mu^k}},1_1\right\rangle_{\GL_{|\muhat|}}=$

$$
\sum_{h\in\calP^\calO,|\omega_h|=|\muhat|}\left((-1)^{f(\omega_h)}H_{\omega_h}(q) q^{\frac{1}{2}|\muhat|(|\muhat|-1)-n(\omega_h)}\right)^{2g+k-2}q^{-\sum_in(\mu^i)}(-1)^{k|\omega_h|-kf(\omega_h)}\sum_{\lambdahat\unlhd\muhat}\prod_{i=1}^k\tilde{K}_{\mu^i\lambda^i}(q)\calH_{\lambda^i}(q)\langle\tilde{H}_{\lambda^i}(\x_i;q),s_{\omega_h}(\x_i)\rangle$$

\begin{align*}
&=\sum_{h\in\calP^\calO,|\omega_h|=|\muhat|}\left(H_{\omega_h}(q) q^{\frac{1}{2}|\muhat|(|\muhat|-1)-n(\omega_h)}\right)^{2g+k-2}q^{-\sum_in(\mu^i)}(-1)^{k|\omega_h|}\prod_{i=1}^k\sum_{\lambda\unlhd\mu^i}\tilde{K}_{\mu^i\lambda}(q)\calH_{\lambda}(q)\langle\tilde{H}_{\lambda}(\x_i;q),s_{\omega_h}(\x_i)\rangle\\
&=q^{\frac{d_\muhat}{2}-1}\sum_{h\in\calP^\calO,|\omega_h|=|\muhat|}q^{(1-g)|\muhat|}\left(H_{\omega_h}(q) q^{-n(\omega_h)}\right)^{2g+k-2}\prod_{i=1}^k\langle s_{\mu^i{'}}(\x_i),s_{\omega_h}(\x_i\y)\rangle
\end{align*}
The last equality follows from Proposition \ref{propmagic} (see also the calculation after Proposition \ref{propmagic}).

We thus have 
\bigskip

$
\sum_\muhat q^{1-d_\muhat/2}\left\langle\calE*\bfX_{\overline{C}_{\mu^1}}*\cdots*\bfX_{\overline{C}_{\mu^k}}, 1_1\right\rangle_{\GL_{|\muhat|}}s_{\muhat'}$

\begin{align*}
&=\sum_{h\in\calP^\calO}q^{(1-g)|\omega_h|}\left(H_{\omega_h}(q)q^{-n(\omega_h)}\right)^{2g+k-2}\prod_{i=1}^ks_{\omega_h}(\x_i\y)\\
&=\prod_{\gamma\in\calO}\sum_\lambda q^{(1-g)|\gamma|\,|\lambda|}\left(H_\lambda(q^{|\lambda|})q^{|\gamma|n(\lambda)}\right)^{2g+k-2}\prod_{i=1}^ks_\lambda(\x_i^{|\gamma|}\y^{|\gamma|})\\
&=\prod_{\gamma\in\calO}\Omega\left(\x_1^{|\gamma|},\dots,\x_k^{|\gamma|};q^{|\gamma|/2},q^{-|\gamma|/2}\right)\\
&=\prod_{d=1}^\infty\Omega\left(\x_1^d,\dots,\x_k^d;q^{d/2},q^{-d/2}\right)^{\phi_d(q)}.
\end{align*}
where $\phi_d(q)$ is the number of $\gamma\in\calO$ of size $d$. 

Now taking Log on both sides we find

\begin{align*}
\Log\,\left(\sum_\muhat q^{1-d_\muhat/2}\left\langle\calE*\bfX_{\overline{C}_{\mu^1}}*\cdots*\bfX_{\overline{C}_{\mu^k}}, 1_1\right\rangle_{\GL_{|\muhat|}}s_{\muhat'}\right)&=(q-1)\,\Log\,\Omega(\sqrt{q},1/\sqrt{q})\\
&=\frac{1}{1-q^{-1}}\sum_\muhat\H_\muhat\left(\sqrt{q},\frac{1}{\sqrt{q}}\right)s_{\muhat'}.
\end{align*}
\end{proof}

\section{Examples: Affine cases}

Let $\bC=(C_1,\dots,C_k)$ be a generic tuple of conjugacy classes of $\GL_n(\C)$ such that the underlying graph of $\Gamma_\bC$ is of type $\tilde{D}_4$, $\tilde{E}_6$, $\tilde{E}_7$ or $\tilde{E}_8$ (we must have $g=0$) and that $\v_\bC$ is a positive imaginary root. Then the character variety $\M_{\overline{\bC}}$ is a surface. In what follow we only treat the case $\tilde{D}_4$ as the other cases are similar.

We thus assume that $k=4$ and that the coordinates of the dimension vector $\v=\v_\bC$ are  $2r$, for some integer $r>0$, on the central vertex and $r$ at the edge vertices. Each conjugacy class $C_i$ must be either semisimple of type $(1^r)(1^r)$, i.e. with two distinct eigenvalues both of multiplicity $r$, or of unipotent type $(2^r)$, i.e., it has a unique eigenvalue and $r$ Jordan blocks of size $2$. Now we can verify that if $\M_{\bC'}$ is a stratum of $\M_{\overline{\bC}}$ with $\bC'\neq\bC$, then it is not empty if and only if three of the coordinates of $\bC'$ are of type $(1^r)(1^r)$ or $(2^r)$ and the other one is of unipotent type $(1^2,2^{r-1})$. Indeed if $\bC'$ is of this form then $d_{\bC'}=0$ and so $\v_{\bC'}$  is a real root (see Kac \cite[Proposition page 78]{kac}), and $d_{\bC'}<0$ otherwise.

Put

\begin{align*}
&\omhat_0=((1^r)(1^r),(1^r)(1^r),(1^r)(1^r),(1^r)(1^r))\\
&\omhat_1=((1^r)(1^r),(1^r)(1^r),(1^r)(1^r),(2^r))\\
&\omhat_2=((1^r)(1^r),(1^r)(1^r),(2^r),(2^r))\\
&\omhat_3=((1^r)(1^r),(2^r),(2^r),(2^r))\\
&\omhat_4=((2^r),(2^r),(2^r),(2^r)).
\end{align*}

It is conjectured \cite[Conjecture 1.5.4]{HLV} that $\H_{\omhat_0}(z,w)=z^2+4+w^2$ for all $r>0$. Recall that the complete symmetric functions decompose into Schur symmetric functions as

$$
h_\mu=\sum_\lambda K_{\lambda\mu}s_\lambda,
$$
where $K_{\lambda\mu}$ are the Kostka numbers. The Kostka number $K_{\lambda\mu}$ being the number of \emph{tableaux} of shape $\lambda$ with content $\mu$, it is easy to see that 

$$
h_{(r,r)}=\sum_{s=0}^rs_{(r+s,r-s)}.
$$
If $\omhat'\in\otT$ is the type of $\bC'\unlhd\bC$ as above (i.e. with $d_{\bC'}=0$), then $\M_{\overline{\bC}{'}}=\{pt\}$ and so conjecturally $\H_{\omhat'}(z,w)=1$, and if $\omhat\in\otT$ is the type of $\bfS\unlhd\bC$ such that the stratum $\M_\bfS$ is empty then conjecturally $\H_\omhat(z,w)=0$. Hence we have the following conjectural identity  for each $i=0,1,2,4$

\beq
\H_{\omhat_0}(z,w)\stackrel{?}{=}\H_{\omhat_i}(z,w)+i.
\label{for0-comb}\eeq
from which we deduce the conjectural combinatorial identity

\beq
\H_{\omega_i}(z,w)\stackrel{?}{=}z^2+(4-i)+w^2
\label{for-comb}\eeq for all $r>0$ and all $i=0,1,2,3,4$. 
Conjecture \ref{mainconj} implies thus the following one.

\begin{claim} If $\bC_i$ denotes a generic tuple of type $\omhat_i$, with $i=1,2,3,4$, then 

\beq
IH_c(\M_{\overline{\bC}_i}\,;\,q,t)=t^2+(4-i)qt^2+q^2t ^4.
\label{conj-D4}\eeq
\label{conj1-D4}\end{claim}

 It is known from Oblomkov \cite[Theorem 6.14]{EOR} that $\M_{\bC_o}=\M_{\overline{\bC}_o}$ is isomorphic to $S^o=S\backslash\Delta\subset\mathbb{P}^3$ where $S$ is a smooth cubic surface and $\Delta$ the union of the coordinate axes. It is then not hard to check (using the long exact sequence in cohomology) that $H_c(S^o\,;\, q,t)=t^2+4qt^2+q^2t^4$ proving  that Formula (\ref{conj-D4}) is true for $i=0$ and giving thus a  strong support for Claim \ref{conj1-D4}. Similar results are also available for $\tilde{E}_6,\tilde{E}_7,\tilde{E}_8$ by Etingof-Oblomkov-Rains \cite{EOR}.

Since it is recursive, Formula (\ref{for-comb}) can be checked by straightforward calculation only for small values of $r$. The proposition below gives further evidences.

\begin{proposition} The conjectural formula (\ref{for-comb}) is true  under both specializations $(z,w)\mapsto (0,z)$ and $(z,w)\mapsto (z^{-1},z)$ for all $r$. 
\end{proposition}

\begin{proof}For a finite quiver $\Gamma$ with vector dimension $\v$, denote by $A_{\Gamma,\v}(q)$ the polynomial that counts the number of isomorphism classes of absolutely indecomposable representations of $(\Gamma,\v)$ over a finite field $\F_q$. We know \cite[Theorem 1.3.1]{HLV2} that $\H_{\omhat_0}(0,\sqrt{q})$ equals $A_{\tilde{D}_4,\v}(q)$ where $\v$ has coordinate $2r$ on the central vertex and $r$ on the edge vertices. By Kac \cite{kac}, the   polynomial $A_{\Gamma,\v}(q)$ is monic of degree $1-\frac{1}{2}{^t}\v C\v$ (where $C$ is the Cartan matrix of the quiver $\Gamma$) and by Hausel \cite{hausel} we know that the constant term of $A_{\Gamma,\v}(q)$ equals the multiplicity of the dimension vector $\v$ (this was conjectured by Kac). Hence $A_{\tilde{D}_4,\v}(q)=q+4$ as the multiplicity of the imaginary roots of $\tilde{D}_4$ equals $4$ by Kac \cite[Corollary 7.4]{kac-book}. We thus proved that Formula (\ref{for-comb}) is true under the specialization $(z,w)\mapsto (0,w)$ when $i=0$. For $i=1,2,3,4$, we are reduced to prove that Formula (\ref{for0-comb}) is true under $(z,w)\mapsto (0,w)$, i.e., that $\H_\omhat(0,w)=1$ if and only if the dimension vector $\v_\omhat$ (as defined in \S \ref{ne}) is a real root and that $\H_\omhat(0,w)=0$ if $\v_\omhat$ is not a root, but  this is \cite[Theorem 23]{letellier5}.
Let us now discuss the other specialization $(z,w)\mapsto (z^{-1},z)$. By the comment below Claim \ref{conj1-D4}, it is clear that $\H_{\omhat_0}(z^{-1},z)=z^{-2}+4+z^2$. To verify that Formula (\ref{for-comb}) is true under the specialization $(z,w)\mapsto (z^{-1},z)$ for $i=1,2,3,4$, it is thus enough to verify that Formula (\ref{for0-comb}) is true under $(z,w)\mapsto(z^{-1},z)$, but this follows from Theorem \ref{maintheo1}.
\end{proof}

\end{document}